\documentclass[a4paper,11pt,reqno]{amsart}

\usepackage[utf8]{inputenc}
\usepackage[T1]{fontenc}			
\usepackage[english]{babel}
\usepackage{amsmath, amsfonts, amssymb, amsthm, amscd, mathtools, mathrsfs}
\usepackage{geometry}
\usepackage[square,numbers]{natbib}	
\usepackage{bbm, bm}
\usepackage{xspace} 
\usepackage{enumerate}
\usepackage{hyperref}
\usepackage{xcolor}

\newcommand{\bbE}{{\ensuremath{\mathbb E}} }
\newcommand{\bbF}{{\ensuremath{\mathbb F}} }

\newcommand{\bbP}{{\ensuremath{\mathbb P}} }

\newcommand{\cA}{{\ensuremath{\mathcal A}} }

\newcommand{\cD}{{\ensuremath{\mathcal D}} }
\newcommand{\cE}{{\ensuremath{\mathcal E}} }
\newcommand{\cF}{{\ensuremath{\mathcal F}} }

\newcommand{\cH}{{\ensuremath{\mathcal H}} }

\newcommand{\cK}{{\ensuremath{\mathcal K}} }
\newcommand{\cL}{{\ensuremath{\mathcal L}} }

\newcommand{\cN}{{\ensuremath{\mathcal N}} }

\newcommand{\cU}{{\ensuremath{\mathcal U}} }

\newcommand{\cX}{{\ensuremath{\mathcal X}} }
\newcommand{\cY}{{\ensuremath{\mathcal Y}} }

\newcommand{\bfu}{{\ensuremath{\mathbf u}} }

\newcommand{\bfS}{{\ensuremath{\mathbf S}} }

\newcommand{\bfzero}{{\ensuremath{\mathbf 0}} }

\newcommand{\dd}{{\ensuremath{\mathrm d}} }
\newcommand{\de}{{\ensuremath{\mathrm e}} }

\newcommand{\ds}{{\ensuremath{\mathrm s}} }

\newcommand{\dB}{{\ensuremath{\mathrm B}} }
\newcommand{\dC}{{\ensuremath{\mathrm C}} }
\newcommand{\dD}{{\ensuremath{\mathrm D}} }

\newcommand{\dL}{{\ensuremath{\mathrm L}} }

\newcommand{\sB}{{\ensuremath{\mathscr B}} }

\newcommand{\Nn}{\mathcal N}

\newcommand{\R}{\mathbb{R}}

\newcommand{\N}{\mathbb{N}}

\newcommand{\E}{\mathbb{E}}

\newcommand{\ind}{\ensuremath{\mathbf{1}}}

\newcommand{\dUC}{{\ensuremath{\mathrm{UC}}} }
\newcommand{\suptwo}[2]{\sup_{\substack{#1 \\ #2}}} 
\DeclarePairedDelimiter{\abs}{\lvert}{\rvert}
\DeclarePairedDelimiter{\norm}{\lVert}{\rVert}
\DeclarePairedDelimiterX{\inprod}[2]{\langle}{\rangle}{#1, #2}
\DeclarePairedDelimiter{\prodscal}{\langle}{\rangle}
\DeclareMathOperator*{\Tr}{Tr}
\DeclareMathOperator*{\im}{Im}

\DeclareMathOperator*{\klim}{\mathcal{K}--\mathrm{lim}}

\renewcommand{\epsilon}{\varepsilon}
\newcommand{\changetheta}
{
\let\temp\theta
\let\theta\vartheta
\let\vartheta\temp
}
\newcommand{\changephi}
{
\let\temp\phi
\let\phi\varphi
\let\varphi\temp
}

\theoremstyle{plain}
\newtheorem{theorem}{Theorem}[section]
\newtheorem*{theorem*}{Theorem}
\newtheorem{lemma}[theorem]{Lemma}
\newtheorem*{lemma*}{Lemma}
\newtheorem{proposition}[theorem]{Proposition}
\newtheorem*{proposition*}{Proposition}
\newtheorem{corollary}[theorem]{Corollary}

\theoremstyle{definition}
\newtheorem{definition}[theorem]{Definition}

\newtheorem{hypothesis}[theorem]{Hypothesis}
\newtheorem{example}[theorem]{Example}

\theoremstyle{remark}
\newtheorem{remark}[theorem]{Remark}
\newtheorem*{remark*}{Remark}

\definecolor{darkviolet}{rgb}{0.58, 0.0, 0.83}

\definecolor{gre}{rgb}{0.03,0.50,0.03}

\definecolor{blu}{rgb}{0.0,0.0,1.0}

\numberwithin{equation}{section}

\allowdisplaybreaks

\hypersetup{pdftitle={Existence of mild solutions for HJB equations on half-spaces of Hilbert spaces},pdfauthor={Alessandro Calvia, Gianluca Cappa, Fausto Gozzi, Enrico Priola},pdftex,%
pdfstartview={XYZ null null 1.4},colorlinks=true,linkcolor=blue,urlcolor=blue}

\newcommand{\mail}[1]{\href{mailto:#1}{\normalfont\texttt{#1}}}

\makeatletter
\def\@setthanks{\vspace{-\baselineskip}\def\thanks##1{\@par##1\@addpunct.}\thankses}
\makeatother

\def\ds{\begin{displaystyle}}
\def\eds{\end{displaystyle}}

\def\<{\left\langle }
\def\>{\right\rangle }

\title[HJB equations on half-spaces]{HJB equations and stochastic control on half-spaces of Hilbert spaces}

\author[A.~Calvia]{Alessandro Calvia\textsuperscript{\MakeLowercase{a},1}}
\thanks{\noindent \textsuperscript{a} Dipartimento di Economia e Finanza, LUISS University, Roma.}

\author[G.~Cappa]{Gianluca Cappa\textsuperscript{\MakeLowercase{a},2}}

\author[F.~Gozzi]{Fausto Gozzi\textsuperscript{\MakeLowercase{a},3}}

\author[E.~Priola]{Enrico Priola\textsuperscript{\MakeLowercase{b},4}}
\thanks{\noindent \textsuperscript{b} Dipartimento di Matematica "Felice Casorati", University of Pavia.
\\
\noindent \textsuperscript{1} E-mail: \mail{acalvia@luiss.it}
\\
\noindent \textsuperscript{2} E-mail: \mail{gcappa@luiss.it}
\\
\noindent \textsuperscript{3} E-mail: \mail{fgozzi@luiss.it}
\\
\noindent \textsuperscript{4} E-mail: \mail{enrico.priola@unipv.it}}

\begin{document}

\begin{abstract}
In this paper we study a first extension of the theory of mild solutions for HJB equations in Hilbert spaces to the case when the domain is not the whole space.
More precisely, we consider a half-space as domain, and a semilinear Hamilton-Jacobi-Bellman (HJB) equation. Our main goal is to establish the existence and the uniqueness of solutions to such HJB equations, that are continuously differentiable in the space variable. We also provide an application of our results to an exit time optimal control problem and we show that the corresponding value function is the unique solution to a semilinear HJB equation, possessing sufficient regularity
to express the optimal control in feedback form.
Finally, we give an illustrative example.
\end{abstract}

\maketitle


\bigskip

\noindent \emph{AMS 2020:}
35R15 (PDEs on infinite-dimensional (e.g., function) spaces),
47D07 (Markov semigroups and applications to diffusion processes),
49L12 (Hamilton-Jacobi equations in optimal control and differential games),
49L20 (Dynamic programming in optimal control and differential games),
93E20 (Optimal stochastic control).

\bigskip

\tableofcontents

\section{Introduction}\label{sec:introduction}

A typical and nontrivial feature of optimal control problems in real applications (both in the deterministic and in the stochastic case) is the fact that the state variable must satisfy suitable constraints, i.e., to belong to a given subset $D$ of its state space $H$.
In some settings, as in the case of state constrained problems (see, e.g,~\citet[Chapter IV]{BardiCapuzzoDolcetta97}), one restricts the set of admissible control strategies and considers only those that keep the state inside the given subset $D$. In some other settings, as in exit time optimal control problems (see, e.g.,~\citet[Chapter 8]{CannarsaSinestrari04}), one does not operate such restriction but terminates the control actions (and, thus, stops the optimization problem) once the state goes out of $D$.

In both cases, standard arguments involving the dynamic programming principle provide a \emph{Hamilton-Jacobi-Bellman} (HJB) equation associated to the optimal control problem, which is a Partial Differential Equation (PDE) in the domain $D$, satisfying suitable boundary conditions.
Such type of HJB equations have been studied in many papers in the finite dimensional case, finding theorems on existence and uniqueness of viscosity solutions (see, e.g.,~\citet{soner:optcontrol1,soner:optcontrol2}) and, in some cases, results on their regularity
(see, e.g.,~\citet[Section 9]{CrandallKocanSwiech00}).
In the infinite dimensional case, the theory of existence and uniqueness of viscosity solutions
of HJB equations in domains has been studied as well, both in the deterministic
(see, e.g.,~\citet[Chapter 4]{LY})
and in the stochastic case (see, e.g.,~\citet[Chapter 3]{fabbri:soc}).

However, in general, solutions to HJB equations in the viscosity sense are not regular enough to find the optimal strategies of a control problem. Said otherwise, if one wants to find optimal strategies it is fundamental to establish results on the regularity of solutions to such HJB equations, showing that they are, at least, continuously differentiable in the state variable.
To the best of our knowledge, these results seem to be completely missing in the literature in a stochastic framework, except for some cases where explicit solutions can be found (see e.g.~\citet[Section 4.10]{fabbri:soc},
\citet{DaPratoZabczyk97}, \citet{BiffisGozziProsdocimi20}, \citet{GozziLeocata22}).

The aim of this paper is twofold: we want to establish, first, a result on the existence and the uniqueness of regular solutions (i.e., continuously differentiable in the state variable) of second-order infinite dimensional semilinear HJB equations in domains; then, we want to prove a verification theorem for the associated exit time stochastic optimal control problem, using the aforementioned existence and uniqueness result.
To do so, we extend to the case of domains the theory of mild solutions for second-order semilinear HJB equations in Hilbert spaces (for a summary on this theory see, e.g.,~\citet[Chapter 4]{fabbri:soc}).
As a starting point for future research, in this paper we consider the domain $D$ to be a half-space.

\subsection{Methodology and main results}
The starting point of our paper are global gradient estimates up to the boundary  for solutions to second-order linear PDEs in special  half-spaces of Hilbert spaces, see  Priola~\cite{Priola2001:schauder} and Chapter 5 in Priola~\cite{Priola99}.
 We also mention~\citet{Priola2003:dirichlet} for related finite-dimensional Dirichlet problems.
 We point out that in    general  domains of Hilbert spaces    local regularity  results are obtained in \cite{dapratogoldyszab96} and  \cite{talarczyk00} (see also \citet[Chapter 8]{DaPrato2002:2PDEs}). On the other hand,
 global gradient estimates for Ornstein-Uhlenbeck Dirichlet semigroups can fail even in a half space
 of $\R^2$  (see an example in \cite{wang2004}).
Using suitable extension operators from the half-space to the whole Hilbert space, we define a family of operators on the set of bounded measurable functions on the half-space. We show that this family is a semigroup of contractions, which coincides with the semigroup defined in~\citet[Chapter 8]{DaPrato2002:2PDEs}, providing the \mbox{so-called} \emph{generalized solution} to these second-order linear PDEs. We also prove some regularizing properties of this semigroup.

Next, we turn our attention to establishing existence and uniqueness of mild solutions (i.e., solutions in integral form) of semilinear HJB equations in the half-space. This result, given in Theorem~\ref{thm:esis_unic_sol_mild}, extends analogous ones provided in~\citet[Chapter 4]{fabbri:soc}, which were proved in the whole Hilbert space case.

As stated previously, we need more regular solutions to be able to find optimal strategies. Thus, in Theorem~\ref{thm:esist_strong_sol} we show that mild solutions are indeed strong solutions, in the sense they can be approximated by classical solutions. More precisely, we rely on the concept of $\cK$-convergence, see Definition~\ref{def:Kconv}.

These results can be profitably applied to study a family of stochastic optimal control problems with exit time and they permit us to prove a verification theorem (see Theorem~\ref{th:ver}). This is a sufficient condition of optimality which allows to write the optimal control in feedback form, when a solution to the so-called closed-loop equation~\eqref{eq:cle} can be found.

\subsection{Plan of the paper}

The plan of the paper is the following.
\begin{itemize}
\item Section \ref{SE:PREL} contains some important preliminary material
on second-order linear PDEs, which are needed to prove our main results.
It is divided in four subsections:
\begin{itemize}
  \item  Subsection \ref{SSE:NOTATION}, where we recall some basic notation used in this paper;
  \item  Subsection \ref{SSE:LINEAR}, where we present some results on second-order linear PDEs in half-spaces;
  \item  Subsection \ref{SSE:SPACES}, where we define some function spaces needed in the main results;
  \item  Subsection \ref{SSE:REGP}, where we prove a regularization result for the semigroup associated to linear PDEs in half-spaces.
\end{itemize}

\item In Section~\ref{sec:mildsol} we establish our first main result, namely, Theorem~\ref{thm:esis_unic_sol_mild} on the existence and uniqueness of mild solution for HJB equation~\eqref{HJB_2}.
\item In Section~\ref{sec:strongsol} we prove our second main result, i.e., Theorem~\ref{thm:esist_strong_sol}, which shows that the mild solution of HJB equation~\eqref{HJB_2} is also a $\mathcal{K}$-strong solution, that is, it can be approximated by classical solutions.
\item In Section~\ref{sec:optctrl} we establish our third main result, namely, the Verification Theorem~\ref{th:ver}, providing a sufficient condition of optimality for the optimal control problem~\eqref{eq:valuefunct}.
\end{itemize}

\section{Preliminaries}
\label{SE:PREL}

\subsection{Notation and basic spaces}
\label{SSE:NOTATION}
In this section we collect the main notations and conventions used in this research article.

Throughout the paper the set $\N$ denotes the set of natural integers $\N \coloneqq \{1, 2, \dots \}$, and the symbol $\R$ denotes the set of real numbers, equipped with the usual Euclidean norm $\abs{\cdot}_\R$. We set $\R_+ \coloneqq (0,+\infty)$. The symbol $\ind_A$ denotes the indicator function of a set $A$.

If $X$, $Y$ are two Banach spaces, the symbols $\dB_b(X;Y)$, $\dC_b(X;Y)$, $\dUC_b(X;Y)$ indicate the sets of $Y$-valued bounded Borel measurable, continuous, uniformly continuous functions on $X$, respectively. In what follows, we will just write \emph{measurable} instead of \emph{Borel measurable}. We denote by $\norm{\cdot}_{\dB_b(X;Y)}$ (resp., $\norm{\cdot}_{\dC_b(X;Y)}$, $\norm{\cdot}_{\dUC_b(X;Y)}$) the usual sup-norm, making $\dB_b(X;Y)$ (resp., $\dC_b(X;Y)$, $\dUC_b(X;Y)$) Banach spaces. If $Y = \R$, we will simply write $\dB_b(X)$, $\dC_b(X)$, $\dUC_b(X)$. The symbol $\dUC^1_b(X)$ denotes the set of all real-valued functions on $X$, that are uniformly continuous and bounded on $X$ together with their \mbox{first-order} Fréchet derivatives.

Throughout the paper, $H$ is a real separable Hilbert space, with inner product $\inprod{\cdot}{\cdot}$ and associated norm $\abs{\cdot}$. The symbol $\cL(H)$ denotes the Banach space of all linear and bounded operators from $H$ into itself, endowed with the supremum norm $\norm{\cdot}_{\cL(H)}$, and $\cL^+(H)$ indicates the subset of $\cL(H)$ consisting of all positive and self-adjoint operators. $\cL_1(H)$ indicates the set of all \emph{trace class} (or \emph{nuclear}) operators from $H$ into itself and $\cL^+_1(H) \coloneqq \cL_1(H) \cap \cL^+(H)$. If $T \in \cL_1(H)$, we define
\begin{equation*}
\Tr T \coloneqq \sum_{k \in \N} \inprod{T e_k}{e_k},
\end{equation*}
where $\{e_k\}_{k \in \N}$ is a complete orthonormal system for $H$.
The symbol $\cN(x,B)$ denotes the Gaussian measure on $H$ with mean $x \in H$ and covariance operator $B \in \cL^+_1(H)$.
The Gaussian measure on $\R$ with mean $\mu \in \R$ and variance $\sigma^2 \geq 0$, is indicated by $\cN_1(\mu, \sigma^2)$.

If $f \colon [0,+\infty) \times H \to \R$ is a differentiable function, $f_t$ denotes the derivative of $f(t,x)$ with respect to $t$ and $\dD f$, $\dD^2 f$ denote the first- and second-order Fréchet derivatives of $f(t,x)$ with respect to $x$, respectively.

From this point onward, let $\{\bar{y},e_k\}_{k\geq2}$ be a fixed orthonormal basis of $H$. We consider the open half-space of $H$ generated by $\bar y$, namely,
\begin{equation*}
H_+ \coloneqq \{x\in H \colon \prodscal{x,\bar{y}}>0\},
\end{equation*}
and we also define the sets
\begin{align*}
H_- &\coloneqq \{x\in H \colon \prodscal{x,\bar{y}}<0\},
\\
\partial H_+ &\coloneqq \{x\in H \colon \prodscal{x,\bar{y}}=0\}.
\end{align*}

We will always identify any element $x \in H$ with the sequence of Fourier coefficients $(x_k)_{k \in \N}$ with respect to $\{\bar y, e_k\}_{k \geq 2}$, where $x_1 \coloneqq \prodscal{x,\bar{y}}$ and $x_k \coloneqq \prodscal{x,e_k}$, for $k\geq2$.
It is important to recall that this identification defines an isometry between $H$ and $\ell_2$, the Hilbert space of real-valued, square-summable sequences.

We will also consider the sub-space $H'$ of $H$ generated by the system $\{e_k\}_{k\geq2}$. As above, we identify any element $x' \in H'$ with the sequence of Fourier coefficients $(x'_k)_{k \geq 2}$, which will be still denoted by $x'$. To be precise, we should identify $x'$ with the sequence $(0, x'_2, x'_3, \ldots)$ but, for the sake of brevity, we will omit the leading zero.
This notation allows to identify any element $x \in H$ with the pair $(x_1, x') \in \R \times H'$, and hence we will write $(x_1, x')$ in place of $x$ whenever necessary. Notice, also, that we have an isometry between $\partial H_+$ and $H'$.

Finally, we introduce the function spaces $\dB_0(\overline{H_+})$, $\dC_0(\overline{H_+})$, $\dUC_0(\overline{H_+})$ indicate, respectively, the sets of bounded measurable, bounded continuous, bounded uniformly continuous functions of $\overline{H_+}$ that vanish on $\partial H_+$.
It is easy to prove that all these three spaces, endowed with the supremum norm, are Banach spaces.
Other spaces, such as $\dB_b(H_+)$, can be defined similarly. We also recall that $\dUC_b(\overline{H_+}) = \dUC_b(H_+)$.

\subsection{The linear problem on the half space}
\label{SSE:LINEAR}
     
In this section we will study the following equation on the closed half-space $\overline{H_+}$
\begin{equation}\label{HJB_condizione_iniziale}
  \begin{dcases}
  v_t(t,x)=\frac{1}{2}\Tr[Q \dD^2 v(t,x)]+ \langle A^* \dD v(t,x), x\rangle, & x\in H_+, \, t>0,\\  
  v(0,x)=\phi(x), & x\in H_+,\\
  v(t,x)=0, & x\in \partial H_+, \, t \geq 0,
  \end{dcases}
\end{equation} 
where $Q\in \cL^+_1(H)$ is a positive, self-adjoint, trace class operator in $H$, $A$ is a linear operator on $H$ (whose adjoint is denoted by $A^*$) and $\phi \in \dB_b(H_+)$ is a given function.
We will work under the following hypothesis, that will stand from now on.
\begin{hypothesis}\label{ipotesi}
\begin{enumerate}[(i)]
\item[]
\item\label{ass:Agenerator} $A \colon D(A)\subset H\to H$ is the generator of a $C_0$-semigroup $\{\de^{tA}\}_{t \geq 0}$;
\item\label{ass:alpha} There exists $\alpha \in\R$ such that $A\bar{y}=\alpha\bar{y}$ and $A^* \bar y = \alpha \bar y$;
\item\label{ass:lambda} There exists $\lambda \in \R_+$ such that $Q\bar{y}=\lambda\bar{y}$;
\item\label{ass:commut} The semigroup $\{\de^{tA}\}_{t \geq 0}$ commutes with the operator $Q$;
\item\label{ass:Qt} For all $t > 0$, the operator $Q_t$ defined below is of trace class
\begin{equation}\label{eq:Qt}
Q_t \coloneqq \int_0^t \de^{sA}Q\de^{sA^*} \, \dd s;
\end{equation}
\item\label{ass:gammaQt} There exists $0<\gamma<1$ such that the operator $\int_0^t s^{-\gamma}\de^{sA}Q\de^{sA^*} \, \dd s$ is of trace class for any $t>0$;
\item\label{ass:rangeincl} For all $t>0$, $\im \de^{tA}\subset \im Q_t^{1/2}$.
\end{enumerate}
\end{hypothesis}

\begin{remark}
\begin{enumerate}[(i)]
\item[]
\item It may be possible to weaken the assumptions on $A$ and $Q$ as in~\citep{priola2000:degendirichlet, priola2002:dirichletinfinite}. 
\item Since the operator $Q_t$, defined in~\eqref{eq:Qt}, is self-adjoint and non-negative for all $t > 0$, point~\eqref{ass:Qt} of Hypothesis~\ref{ipotesi} implies that $Q_t \in \cL^+_1(H)$, for all $t > 0$.
\item As a consequence of the closed graph theorem, Hypothesis~\ref{ipotesi}-\eqref{ass:rangeincl} implies that the operator\footnote{$Q^{-1/2}_t$ is the pseudoinverse of $Q_t^{1/2}$. For a definition see, e.g., \citep[Definition~B.1]{fabbri:soc}.} $Q^{-1/2}_t\de^{tA}$ is bounded for all $t>0$ (see, e.g., \citep[(4.59)]{fabbri:soc}).
\end{enumerate}
\end{remark}
We will need also the following hypothesis.

\begin{hypothesis} \label{ipotesi_su_Lambda}
Let $T>0$. There exists  $C= C_T >0$ and $\delta\in (0,1)$ such that
\begin{equation*}
\norm{Q^{-1/2}_t\de^{tA}}_{\cL(H)}\leq C t^{-\delta},\quad t>0.
\end{equation*}
\end{hypothesis}

It is useful to recall that in the case where the linear HJB~\eqref{HJB_condizione_iniziale} is formulated in the whole space $H$, namely,
\begin{equation}\label{eq:linHJBonH}
  \begin{dcases}
  v_t(t,x)=\frac{1}{2}\Tr[Q \dD^2 v(t,x)]+\langle A^*\dD v(t,x),x\rangle, & x\in H, \, t>0, \\
  v(0,x)=\phi(x), & x\in H,\\
  \end{dcases}
\end{equation}
with $\phi \in \dB_b(H)$, the mild solution of~\eqref{eq:linHJBonH} is $v(t,x) = T_t\phi(x)$, where $\{T_t\}_{t \geq 0}$ is the semigroup associated to a specific Ornstein-Uhlenbeck process $X$.
More precisely, considering a complete filtered probability space $(\Omega, \cF, (\cF_t)_{t \geq 0}, \bbP)$, supporting a $Q$-Wiener process $W = (W(t))_{t \geq 0}$ (for a precise definition see, e.g., \citep[Definition~4.2]{DaPrato2014:stocheq}), $X$ solves the SDE
\begin{equation}\label{EDS}
\left\{
\begin{aligned}
&\dd X(t) = AX(t) \, \dd t + \dd W(t), \quad t > 0, \\
&X(0) = x.
\end{aligned}
\right.
\end{equation}
Under Hypothesis~\ref{ipotesi}, SDE~\eqref{EDS} admits a unique mild solution, denoted by $\{X(t;x)\}_{t \geq 0}$ to emphasize its dependence on the initial condition $x \in H$, given by
\begin{equation*}
X(t;x) \coloneqq \de^{tA}x +\int_0^t \de^{(t-s)A} \, \dd W(s), \quad t \geq 0.
\end{equation*}
Thanks to Hypothesis~\ref{ipotesi}-\eqref{ass:gammaQt}, $X$ has continuous trajectories (see, e.g., \citep[Theorem~1.152]{fabbri:soc}).
Moreover, the semigroup $\{T_t\}_{t \geq 0}$ has the explicit expression (cf. \citep[(4.50)]{fabbri:soc})
\begin{equation} \label{def:T_t}
T_t f(x) \coloneqq \bbE[f(X(t;x))] = \int_H f(y) \, \Nn(\de^{tA}x,Q_t)(\dd y), \quad f\in \dB_b(H).
\end{equation}

In the literature (see for instance \citep[Chapter 8]{DaPrato2002:2PDEs}) the generalized solution to~\eqref{HJB_condizione_iniziale} is defined by
\begin{equation}\label{eq:gensolHJBlin}
v(t,x) = M_t\phi(x) \coloneqq \E\left[\phi(X(t;x))\ind_{\{\bm{\tau}_x>t\}}\right], \quad t \geq 0, \, x\in H_+.   
\end{equation}
where 
\begin{equation*}
\bm{\tau}_x \coloneqq \inf\{t > 0 \colon X(t;x)\in \overline{H_{-}}\}.
\end{equation*}
The purpose of this section is to show that it is possible to define a semigroup $\{P_t\}_{t \geq 0}$ that allows to provide the mild solution to~\eqref{HJB_condizione_iniziale}. We will also show that $\{P_t\}_{t \geq 0}$ coincides with $\{M_t\}_{t \geq 0}$ (called \emph{restricted semigroup} in \citep[Chapter 8]{DaPrato2002:2PDEs}) on a suitable space. Moreover, we will prove some regularizing properties.
The idea (cf. \citep{priola2000:degendirichlet,Priola2001:schauder,priola2002:dirichletinfinite,Priola2003:dirichlet}) is to define suitable extension operators, so that it is possible to exploit the semigroup $\{T_t\}_{t \geq 0}$, given in~\eqref{def:T_t}.

For any $\eta\in   {\dB_b(\overline {H_+})}  $
and $x\in H$ (identified with $(x_1, x') \in \R \times H'$), we set
\begin{equation}\label{def:estensione}
E\eta(x) \coloneqq
\begin{dcases}
  \eta(x_1,x'),\quad &x_1 \ge  0,\\
  -\eta(-x_1,x'), \quad & x_1 <  0.
\end{dcases}
\end{equation}
By abuse of notation, we can adopt the same symbol to denote the following extension operator, defined for any $\eta\in   \dB_b(H_+)$ and $x\in H$,
\begin{equation}\label{def:estensioneap}
E\eta(x) \coloneqq
\begin{dcases}
	\eta(x_1,x'),\quad &x_1 >  0,\\
	0, \quad &x_1=0,\\
  	-\eta(-x_1,x'), \quad & x_1 <  0.
\end{dcases}
\end{equation}
Clearly, $E\eta \in \dB_b(H)$, in both cases.

We need to define a suitable counterpart in $H'$ of the semigroup $\{\de^{tA}\}_{t \geq 0}$. This can be easily done thanks to the following lemma.

\begin{lemma}\label{lem:sgrinv}
For each $t \geq 0$, the operator $\de^{tA}$ leaves $\partial H_+$ invariant.
\end{lemma}
\begin{proof}
Fix $t \geq 0$. By~Hypothesis~\ref{ipotesi}-\eqref{ass:alpha} $\bar y$ is an eigenvector of $A^*$, with eigenvalue $\alpha$. Therefore, we get that, for all $x \in \partial H_+$,
\begin{equation*}
\inprod{\de^{tA}x}{\bar y} = \inprod{x}{\de^{tA^*}\bar y} = \de^{\alpha t} \inprod{x}{\bar y} = 0,
\end{equation*}
whence the claim.
\end{proof}
Using the isometry between $\partial H_+$ and $H'$, here denoted by $\gamma \colon \partial H_+ \to H'$, we can define the family of operators $\tilde S(t) \colon H' \to H'$, given by
\begin{equation*}
\tilde S(t)x' \coloneqq \gamma\left(\de^{tA}(0,x')\right), \quad x' \in H', \, t \geq 0.
\end{equation*}
It is immediate to prove that each of the operators $\tilde S(t)$, $t \geq 0$, is linear and bounded. We also have the following result, whose proof is omitted because it is a standard consequence of the invariance property proved in Lemma~\ref{lem:sgrinv}.

\begin{lemma}
The family of operators $\{\tilde S(t) \}_{t\geq0}$ is a semigroup on $H'$.
\end{lemma}

Thanks to Hypothesis~\ref{ipotesi}-\eqref{ass:lambda}, also the operator $Q$ leaves $\partial H_+$ invariant.
Therefore, we can define the operator $\tilde Q \coloneqq H'\rightarrow H'$ as:
\begin{equation*}
\tilde Q x' \coloneqq \gamma\left(Q(0,x')\right), \quad x' \in H'.
\end{equation*}
Finally, we define
\begin{equation*}
\tilde Q_t \coloneqq \int_0^t \tilde S(s)\tilde Q\tilde S^*(s) \, \dd s, \quad t > 0.
\end{equation*}
Thanks to the isometry $\gamma$ and recalling that $Q \in \cL^+_1(H)$, it is easy to show that also $\tilde Q \in \cL^+_1(H)$. Moreover, considering in addition Hypothesis~\ref{ipotesi}-\eqref{ass:gammaQt}, we have that $\tilde Q_t \in \cL^+_1(H)$, for all $t > 0$.

\begin{remark}
By Hypothesis~\ref{ipotesi}-\eqref{ass:alpha}-\eqref{ass:lambda}, we get
\begin{equation*}
Q_t\bar y = \int_0^t \de^{sA}Q\de^{sA^*}\bar y \, \dd s= g(t) \bar y, \quad t > 0,
\end{equation*}
where
\begin{equation*}
g(t) \coloneqq \lambda \int_0^t \de^{2\alpha s} \, \dd s = \dfrac{\lambda}{2\alpha}(\de^{2\alpha t}-1), \quad t > 0.
\end{equation*}
We deduce that, for any $x\in H$, $x=(x_1,x')$, the Gaussian measure $\Nn(\de^{tA}x,Q_t)$ can be split as follows
\begin{equation*}
\Nn(\de^{tA}x,Q_t)=\Nn_1(\de^{\alpha t}x_1,g(t))\otimes \Nn(\tilde S(t)x',\tilde Q_t), \quad t > 0,
\end{equation*}
where the Gaussian measure on $H'$ is still denoted by $\cN$. 
\end{remark}
 The following lemma provides some useful formulas (cf.~\citep{Priola2003:dirichlet} for the finite-dimensional case).
\begin{lemma}\label{lem:explicitsemigroupT}
Let us define, for all $t \geq 0$, $\theta \in \R$, and $\xi \in \R_+$
\begin{equation}\label{eq:G}
G(t,\theta,\xi) \coloneqq \dfrac{1}{\sqrt{2\pi g(t)}}\left[\exp\left\{-\dfrac{(\theta\de^{\alpha t}-\xi)^2}{2g(t)}\right\} - \exp\left\{-\dfrac{(\theta\de^{\alpha t}+\xi)^2}{2g(t)}\right\}\right].
\end{equation}
Then, for any $\eta\in \dB_b({H_+})$ and any $x\in H$,
\begin{equation}\label{eq:TtE}
T_tE\eta(x)=\int_{H'} \left[\int_{\R_+} G(t, x_1, \xi) \eta(\xi,y')\, \dd\xi\right] \Nn(\tilde S(t)x',\tilde Q_t)(\dd y').
\end{equation}
\end{lemma}

{\allowdisplaybreaks
\begin{proof}
Let $\eta\in \dB_b({H_+})$ and $x=(x_1,x')\in H$, where $x_1\in \R$ and $x'\in H'$. Then,
\begin{align*}
  &\mathrel{\phantom{=}}T_tE\eta(x)=\int_H E\eta(y)\ \Nn(\de^{tA}x,Q_t)(\dd y)\\
  &=\int_{H_+}\eta(y_1,y')\ \Nn(\de^{tA}x,Q_t)(\dd y)-\int_{H_-}\eta(-y_1,y')\ \Nn(\de^{tA}x,Q_t)(\dd y)\\
  &=\int_{H'}\left(\int_0^{+\infty} \eta(\xi,y') \ \Nn_1(\de^{\alpha t}x_1,g(t))(\dd\xi)\right) \Nn(\tilde S(t)x',\tilde Q_t)(\dd y')\\
  &\qquad -\int_{H'} \left(\int_{-\infty}^0 \eta(-\xi,y')\ \Nn_1(\de^{\alpha t}x_1,g(t))(\dd\xi)\right) \Nn(\tilde S(t)x',\tilde Q_t)(\dd y')\\
  &=\int_{H'}\left(\int_0^{+\infty} \eta(\xi,y') \ \Nn_1(\de^{\alpha t}x_1,g(t))(\dd\xi)\right) \Nn(\tilde S(t)x',\tilde Q_t)(\dd y')\\
  &\qquad -\int_{H'} \left(\int_0^{+\infty} \eta(\xi,y')\ \Nn_1(-\de^{\alpha t}x_1,g(t))(\dd\xi)\right) \Nn(\tilde S(t)x',\tilde Q_t)(\dd y')\\
  &=\int_{H'} \left[\int_{\R_+} \dfrac{1}{\sqrt{2\pi g(t)}}\left[\exp\left\{-\dfrac{(\de^{\alpha t}x_1-\xi)^2}{2g(t)}\right\} - \exp\left\{-\dfrac{(\de^{\alpha t}x_1+\xi)^2}{2g(t)}\right\}\right]\eta(\xi,y')\,\dd\xi\right]\times\\
  &\qquad \times\Nn(\tilde S(t)x',\tilde Q_t)(\dd y'). \qedhere
\end{align*}
\end{proof}
}

We introduce, next, a family of operators on $\dB_b(\overline{H_+})$.

\begin{definition}\label{def:Pt}
For any $\eta\in \dB_b(\overline{H_+})$ and any $x\in \overline{H_+}$, we define the family of operators $\{P_t\}_{t \geq 0}$ as
\begin{equation*}
P_t\eta(x) \coloneqq R T_t E\eta(x), \quad t \geq 0,
\end{equation*}
where $Rf$ is the restriction of $f\in \dB_b(H)$ to $\overline {H_+}$.
\end{definition}

\begin{remark}\label{rem:Pt_aperto}
Using the extension defined in~\eqref{def:estensioneap}, we immediately see that the family $\{P_t\}_{t \geq 0}$ is uniquely defined on functions $\eta \in \dB_b({H_+})$. To see this, it is enough to recall that
\begin{equation*}
\Nn(\de^{tA}x,Q_t) (\partial H_+) = 0, \quad t >0, \, x \in H.
\end{equation*}
Note that for this reason, while in~\eqref{def:estensioneap} we choose to extend $\eta$ to be null on $\partial H_+$, one can opt for an arbitrarily different measurable extension on the boundary of $H_+$ and still obtain a well-defined family $\{P_t\}_{t \geq 0}$ on $\dB_b({H_+})$. It is also worth remembering that, in any case, $P_t \eta$ is a function defined on $\overline{H_+}$, for all $t \geq 0$.
\end{remark}

\begin{proposition}\label{prop_P_t}
The family of operators $\{P_t\}_{t \geq 0}$ is a semigroup of contractions on $\dB_b(\overline{H_+})$. Moreover,
\begin{enumerate}[(i)]
  \item\label{prop:PtUC0} $P_t(\dB_b(\overline {H_+}))\subset \dUC_0(\overline{H_+})$ and $P_t(\dB_b(H_+))\subset \dUC_0(\overline{H_+})$, for all $t>0$.
  \item For every $f\in  \dC_0(\overline{H_+})$ and $x\in \overline{H_+}$, the map $t\mapsto P_t f(x)$, defined on $[0,+\infty)$ with real values, is continuous.
\end{enumerate}
\end{proposition}
\begin{proof}
We prove, first, the semigroup property for $\{P_t\}_{t \geq 0}$. Applying the definition of the restriction $R$, the fact that $G$, defined in \eqref{eq:G}, is an odd function in the second argument, and using \eqref{eq:TtE}, we get that, for all $\eta \in \dB_b(\overline{H_+})$,
\begin{equation}\label{eq:extrestrsgr}
ER T_tE\eta(x)=T_tE\eta(x), \quad x \in H, \quad t \geq 0.
\end{equation}
From this equality, we obtain that, for all $t,s\geq0$,
\begin{equation*}
P_t P_s \eta= P_t(R T_s E\eta)=R T_t E(RT_sE\eta)=R T_t T_sE\eta=R T_{t+s}E\eta= P_{t+s}\eta,
\end{equation*}
i.e., the semigroup property. Now we can prove the others statements.
\begin{enumerate}[(i)]
\item Let $f \in \dB_b(\overline{H_+})$ (the proof in the case where $f \in \dB_b(H_+)$ is identical).
Then, recalling that $Ef\in \dB_b(H)$ and that $T_t \colon \dB_b(H)\to \dUC_b(H)$, for all $t>0$, we get $P_t f\in \dUC_b(\overline{H_+})$, for all $t>0$.
Now, consider $z=(0,z')\in\partial H_+$. Then, by Lemma~\ref{lem:explicitsemigroupT}, since $G(t,0,\xi)=0$, for all $t \geq 0$ and all $\xi \in \R_+$, we get
\begin{equation*}
T_tEf(z)=\int_{H'} \left[\int_{\R_+} G(t, 0, \xi) f(\xi,y')\, \dd\xi\right] \Nn(\tilde S(t)z',\tilde Q_t)(\dd y')=0,
\end{equation*}
whence $P_tf(z) = 0$. Therefore, $P_tf \in \dUC_0(\overline{H_+})$, for all $t > 0$.
\item If $x \in H_+$, the continuity in $t$ of $P_t f(x)$ follows from the continuity of $T_t f(x)$. If $x \in \partial H_+$, it suffices to observe that $P_t f(x) = 0$, for all $t \geq 0$, since $f(x) = 0$ and thanks to point~\eqref{prop:PtUC0}. \qedhere
\end{enumerate}
\end{proof}

We are now ready to show the following important result on semigroups $\{P_t\}_{t \geq 0}$ and $\{M_t\}_{t \geq 0}$.

\begin{proposition}
\label{identita_semigruppi}
On $\dB_b(H_+)$ the identity $P_t=M_t$ holds true for any $t\geq0$.
\end{proposition}

\begin{proof}
First, we set a useful notation. For any $x=(x_1,x')\in H_+$, where $x_1\in \R_+$ and $x'\in H'$, we can write the solution $X(t,x)$ to SDE~\eqref{EDS} as
\[X(t,x) \coloneqq (X^1(t,x_1),X'(t,x')),\]
where $X^1(t,x_1)$ is a stochastic process with values in $\R_+$ while $X'(t,x')$ is a stochastic process with values in $H'$. Therefore
\[
\bm{\tau}_x=\inf\{t > 0 \colon X(t, x)\in \overline{H_{-}}\}=\inf\{t > 0 \colon X^1(t, x_1)=0\} \eqqcolon \tau_1.
\]

Let $f\in \dB_b(H_+)$. Then, using the tower property, we get
\begin{align*}
M_tf&(x)=\E[f(X(t,x))\ind_{\{\bm{\tau}_x>t\}}]=\E[f(X^1(t,x_1),X'(t,x'))\ind_{\{\bm{\tau}_x>t\}}]\\
=&\E\left[\E\left[f(X^1(t,x_1),X'(t,x'))\ind_{\{\tau_1>t\}} \mid X'(t,x')\right] \right].
\end{align*}
Noting that $X^1(t,x_1)$ and $X'(t,x')$ are independent (see \citep[Proposition 2.12]{DaPrato2014:stocheq}), we have
\begin{align*}
M_tf&(x)=\E\left[\int_0^\infty\left(\frac{\de^{-(x_1-\xi)^2/(2g(t))}-\de^{-(x_1+\xi)^2/(2g(t))}}{\sqrt{2\pi g(t)}}\right)f(\xi,y) \, \dd\xi \, \Biggm| \, y=X'(t,x')\right]\\
&=\int_{H'}\left[ \int_0^\infty\left(\frac{\de^{-(x_1-\xi)^2/(2g(t))}-\de^{-(x_1+\xi)^2/(2g(t))}}{\sqrt{2\pi g(t)}}\right)f(\xi,y) \, \dd\xi \right] \Nn(\tilde S(t)x',\tilde Q_t)(\dd y) \\
&=P_tf(x). \qedhere
\end{align*}
\end{proof}

\subsection{Some useful function spaces}
\label{SSE:SPACES}

In this section we introduce some further function spaces, that will be needed in the sequel.
Let $T>0$ and set
\begin{align*}
\dB_0([0,T] \times \overline{H_+}) &\coloneqq \{f \in \dB_b([0,T] \times \overline{H_+}) \text{ s.t. } \psi(t,x) = 0, \, \forall t \in [0,T], \, x \in \partial H_+\}, \\
\dC_0([0,T] \times \overline{H_+}) &\coloneqq \{f \in \dC_b([0,T] \times \overline{H_+}) \text{ s.t. } f(t,x) = 0, \, \forall t \in [0,T], \, x \in \partial H_+\}.
\end{align*}
These are Banach spaces when endowed with the supremum norm and, clearly, the latter is a subspace of the former.

Next, for $\delta\in(0,1)$ as in Hypothesis~\ref{ipotesi_su_Lambda}, $\cX$ denoting $H$, $\overline{H_+}$, or $H_+$, and $\cY$ indicating either $H$ or $\R$, define
\begin{align*}
\dB_{b,\delta}((0,T]\times \cX;\cY) &\coloneqq \{f \colon (0,T]\times \cX\to \cY\text{ measurable, s.t. } f \in \dB_b([t,T]\times \cX;\cY)\\
&\qquad \, \forall t\in(0,T) \text{ and } \{(t,x) \mapsto t^\delta f(t,x)\} \in \dB_b((0,T]\times \cX;\cY)\},\\
\dC_{b,\delta}((0,T]\times \cX;\cY) &\coloneqq \{f \colon (0,T]\times \cX\to \cY \text{ measurable, s.t. } f \in \dC_b([t,T]\times \cX;\cY) \\
&\qquad \, \forall t\in(0,T) \text{ and } \{(t,x) \mapsto t^\delta f(t,x)\} \in \dC_b((0,T]\times \cX;\cY)\}.
\end{align*}
The latter is a subspace of the former space. If $\cY = \R$ we will simply denote them by $\dB_{b,\delta}((0,T]\times \cX)$ and $\dC_{b,\delta}((0,T]\times \cX)$. We endow them with the following norm, making them Banach spaces:
\begin{equation*}
\norm{f}_{\dB_{b,\delta}} \coloneqq \sup_{(t,x) \in (0,T]\times \cX} t^\delta |f(t,x)|_{\cY}.
\end{equation*}
If $f \in \dC_{b,\delta}((0,T]\times \cX;\cY)$, we will write $\norm{f}_{\dC_{b,\delta}}$.
We also introduce
\begin{align*}
\dB^{0,1}_{b,\delta}([0,T]\times \overline{H_+}) &\coloneqq \{f \in \dB_0([0,T]\times \overline{H_+}) \text{ s.t. there exists } \dD f \in \dB_{b,\delta}((0,T]\times H_+;H)\},\\
\dC^{0,1}_{b,\delta}([0,T]\times \overline{H_+}) &\coloneqq \{f \in \dC_0([0,T]\times \overline{H_+}) \text{ s.t. there exists } \dD f \in \dC_{b,\delta}((0,T]\times H_+;H)\}.
\end{align*}
Also in this case, the latter is a subspace of the former space. We endow them with the norm:
\begin{equation*}
\norm{f}_{\dB^{0,1}_{b,\delta}} \coloneqq \norm{f}_{\dB_b([0,T]\times \overline{H_+})} + \norm{\dD f}_{\dB_{b,\delta}((0,T]\times H_+;H)}.
\end{equation*}

\subsection{Regularization property of $P_t$}
\label{SSE:REGP}

In this section we show the regularization property of the semigroup $\{P_t\}_{t \geq 0}$, introduced in Definition~\ref{def:Pt}, and we provide some joint \mbox{time-space} regularity properties of this semigroup that will be particularly useful in the next sections.

We recall that, under Hypotheses~\ref{ipotesi} and~\ref{ipotesi_su_Lambda}, the semigroup $\{T_t\}_{t \geq0}$ satisfies
\begin{equation}\label{eq:gradTtestimate}
\|\dD T_t f\|_{\dUC_b(H)}\leq \frac{C}{t^\delta}\|f\|_{\dB_b(H)},\quad f\in \dB_b(H), \, t \in (0,T],
\end{equation}
for some positive constant $C = C_T$ independent of $f$ (see, e.g., \citep[Chapter~6]{DaPrato2002:2PDEs}).

\begin{proposition}
\label{prop:regolarita}
Let $f\in \dB_b({H_+})$. Then, $P_t f\in \dUC^1_b(\overline{H_+})$ and there exists $C>0$ independent of $f$ such that
\[|\dD P_t f(x)|\leq  \frac{C}{t^\delta}\|f\|_{\dB_b({H_+})},\]
for all $x\in H_+$ and all $t \in (0,T]$.
\end{proposition}
\begin{proof}
Let $f\in \dB_b({H_+}) $. Then, $Ef\in \dB_b(H)$ and using \eqref{eq:gradTtestimate}
\[\|\dD T_t Ef\|_{\dUC_b(H)} \leq \frac{C}{t^\delta}\|Ef\|_{\dB_b(H)}.\]
Recalling that $P_t=RT_tE$, we have
\[\|\dD P_t f\|_{\dUC_b(H_+)}\leq \|\dD T_t Ef\|_{\dUC_b(H)}\leq \frac{C}{t^\delta}\|Ef\|_{\dB_b(H)}\leq \frac{C}{t^\delta}\|f\|_{\dB_b({H_+})}. \qedhere\]
\end{proof}

\begin{proposition}
\label{prop:continuita_misurabilita}
Let $T>0$ and define the sets
\begin{equation*}
I_0 \coloneqq \{(s,t) \colon 0<s\leq t\leq T\}, \qquad I_1 \coloneqq \{(s,t) \colon 0<s< t\leq T\}.
\end{equation*}
Suppose that Hypotheses~\ref{ipotesi} and~\ref{ipotesi_su_Lambda} are satisfied, for some $\delta \in (0,1)$. Then,
\begin{enumerate}[(i)]
\item\label{prop:contmeas_phi0} For every $\eta\in \dB_b(\overline{H_+})$, the function $\eta^0_P \colon [0,T]\times \overline{H_+}\to \R$, defined as
\begin{equation}\label{eq:eta0P}
\eta^0_P(t,x) \coloneqq P_t \eta(x), \qquad (t,x) \in [0,T]\times \overline{H_+},
\end{equation}
belongs to $\dB_b([0,T] \times \overline{H_+}) \cap \dC_b((0,T]\times \overline{H_+})$ and satisfies $\eta^0_P(t,x) = 0$, for all $(t,x) \in (0,T] \times \partial H_+$.
\item\label{prop:contmeas_psi0} For every $\psi\in \dB_{b,\delta}((0,T]\times \overline{H_+})$, the function $\bar\psi^0_P \colon I_0\times \overline{H_+}\to \R$, defined as
\begin{equation*}
\bar\psi^0_P(t,s,x) \coloneqq  P_{t-s}[\psi(s,\cdot)](x), \quad (t,s,x) \in I_0\times \overline{H_+},
\end{equation*}
is measurable and the function\footnote{Defined to be $0$ for $t=0$, by continuity.} $\psi^0_P \colon [0,T]\times \overline{H_+}\to \R$, defined as
\begin{equation*}
\psi^0_P(t,x) \coloneqq \int_0^t P_{t-s}[\psi(s,\cdot)](x) \, \dd s, \quad (t,x) \in [0,T]\times \overline{H_+},
\end{equation*}
belongs to $\dC_0([0,T] \times \overline{H_+})$.
\normalcolor
\item\label{prop:contmeas_phi1} For every $\eta\in \dB_b(\overline{H_+})$, the function $\eta^1_P \colon (0,T]\times H_+\to H$, defined as
\begin{equation*}
\eta^1_P(t,x) \coloneqq \dD P_t[\eta](x), \quad (t,x) \in (0,T]\times H_+,
\end{equation*}
belongs to $\dC_{b,\delta}((0,T]\times H_+;H)$.
\item\label{prop:contmeas_psi1} For every $\psi\in \dB_{b,\delta}((0,T]\times \overline{H_+})$, the function $\bar\psi^1_P \colon I_1\times H_+\to H$, defined as
\begin{equation*}
\bar\psi^1_P(t,s,x) \coloneqq \dD P_{t-s}[\psi(s,\cdot)](x), \quad (t,s,x) \in I_1\times H_+,
\end{equation*}
is measurable and the function $\psi^1_P \colon (0,T]\times H_+\to H$, defined as
\begin{equation*}
\psi^1_P(t,x) \coloneqq \int_0^t \dD P_{t-s}[\psi(s,\cdot)](x) \, \dd s, \quad (t,x) \in (0,T]\times H_+,
\end{equation*}
belongs to $\dC_{b,\delta}((0,T]\times H_+;H)$.
\end{enumerate}
\end{proposition}

\begin{proof}
We recall, first, the definition of semigroup $\{P_t\}_{t \geq 0}$. For any $ f \in \dB_b(\overline {H_+}) $,
\begin{equation*}
P_t f(x)=R T_t Ef(x), \quad x \in \overline{H_+},
\end{equation*}
where $E \colon \dB_b(\overline{H_+})\to \dB_b(H)$ is the extension operator defined in~\eqref{def:estensione}, $R$ is the restriction to the half-plane $\overline{H_+}$ and $T_t$ is the semigroup defined in~\eqref{def:T_t}. Recall, also, that $P_t f \in \dUC_0(\overline{H_+})$, for all $t > 0$, by Proposition~\ref{prop_P_t}-(\ref{prop:PtUC0}).

To prove the properties listed above, we use the fact that semigroup $\{T_t\}_{t \geq 0}$ verifies the assumptions of Propostion~4.50 and Proposition~4.51 in~\citep{fabbri:soc}
\begin{enumerate}[(i)]
\item Take $\eta \in \dB_b(\overline{H_+})$ and consider its extension $E\eta \in \dB_b(H)$. Combining Proposition~4.50-(i) and Proposition~4.51-(i) of~\citep{fabbri:soc}, we have that $(t,x) \mapsto T_t E \eta(x) \in \dB_b([0,T] \times H) \cap \dC_b((0,T] \times H)$. Therefore, since $\eta^0_P(t,x) = RT_tE \eta(x) = T_tE \eta(x)$, for all $(t,x) \in [0,T] \times \overline{H_+}$, we get that $\eta^0_P \in \dB_b([0,T] \times \overline{H_+}) \cap \dC_b((0,T]\times \overline{H_+})$. Moreover, by  Proposition~\ref{prop_P_t}-(\ref{prop:PtUC0}), we have that $\eta^0_P(t,x) = 0$, for all $0 < t \leq T$ and $x \in \partial H_+$.
\item Fix $\psi\in \dB_{b,\delta}((0,T]\times \overline{H_+})$ and consider, for all $t > 0$, the extension $E\psi(t, \cdot) \in \dB_b(H)$. Since
\begin{equation}\label{eq:extension_timespace}
E\psi(t,x) =
\begin{dcases}
\psi(t,x_1,x'), &x_1 \geq 0, \\
-\psi(t,-x_1,x'), &x_1 < 0,
\end{dcases}
\quad (t,x) = (t,x_1,x') \in (0,T] \times H,
\end{equation}
we immediately deduce that $(t,x) \mapsto E\psi(t,x) \in \dB_b((0,T] \times H)$. Therefore, applying~\citep[Propostion~4.50-(ii)]{fabbri:soc}, we get that the map
\begin{equation*}
(t,s,x) \mapsto T_{t-s}[E\psi(s, \cdot)](x), \quad (t,s,x) \in I_0 \times H,
\end{equation*}
is measurable, and hence, noting that
\begin{equation*}
\bar \psi^0_P(t,s,x) = RT_{t-s}[E\psi(s, \cdot)](x) = T_{t-s}[E\psi(s, \cdot)](x), \quad (t,s,x) \in I_0 \times \overline{H_+},
\end{equation*}
we get that $\bar \psi^0_P$ is measurable.
Next, combining Propostion~4.50-(ii) and Proposition~4.51-(ii) of~\citep{fabbri:soc}, we have that the map $(t,x) \mapsto \int_0^t T_{t-s}[E\psi(s,\cdot](x) \, \dd s$ belongs to $\dB_b([0,T] \times H) \cap \dC_b((0,T] \times H)$, and hence, observing that
\begin{equation*}
\psi^0_P(t,x) = \int_0^t RT_{t-s}[E\psi(s,\cdot](x) \, \dd s = \int_0^t T_{t-s}[E\psi(s,\cdot](x) \, \dd s, \quad (t,x) \in [0,T] \times \overline{H_+},
\end{equation*}
we get that $\dB_b([0,T] \times \overline{H_+}) \cap \dC_b((0,T] \times \overline{H_+})$.
We are, thus, left to show that $\psi^0_P(t,x) = 0$, for all $(t,x) \in [0,T] \times \partial H_+$. This is obvious for $t = 0$. Fix $t > 0$ and note that, by Proposition~\ref{prop_P_t}-(\ref{prop:PtUC0}), $P_{t-s}[\psi(s,\cdot)](x) = 0$, for all $0 \leq s < t$ and all $x \in \partial H_+$. Therefore,
\begin{equation*}
\psi^0_P(t,x) = \int_0^t P_{t-s}[\psi(s,\cdot)](x) \, \dd s = 0, \quad (t,x) \in [0,T] \times \partial H_+,
\end{equation*}
whence the claim.
\item Take $\eta \in \dB_b(\overline{H_+})$ and consider its extension $E\eta \in \dB_b(H)$. By \citep[Proposition~4.51-(iii)]{fabbri:soc}, applied with $G=I$, $U=H$ and $\gamma_G(t)=t^{-\delta}$, we have that $(t,x) \mapsto \dD T_t  E \eta(x) \in \dC_{b,\delta}((0,T] \times H; H)$. Therefore, since $\eta^1_P(t,x) = \dD RT_tE \eta(x) = \dD T_tE \eta(x)$, for all $(t,x) \in (0,T] \times \overline{H_+}$, we get the claim.
\item Arguing as in the proof of point~(\ref{prop:contmeas_psi0}), the claim follows applying Proposition~4.50-(iv) and Proposition~4.51-(iv) of~\citep{fabbri:soc}, with $G=I$, $U=H$, $\gamma_G(t)=\eta(t)=t^{-\delta}$. \qedhere
\end{enumerate}
\end{proof}

\begin{remark}\label{rem:sgr_funz_ap}
Using the extension defined in~\eqref{def:estensioneap}, it is immediate to show that the same results of Proposition~\ref{prop:continuita_misurabilita} hold if one considers functions $\eta \in \dB_b(H_+)$ in points~\eqref{prop:contmeas_phi0} and~\eqref{prop:contmeas_phi1} and functions $\psi\in \dB_{b,\delta}((0,T]\times H_+)$ in points~\eqref{prop:contmeas_psi0} and~\eqref{prop:contmeas_psi1}, respectively.
\end{remark}

\begin{remark}\label{rem:funznullealbordo}
Recall that the semigroup $\{P_t\}_{t \geq 0}$ is uniquely defined on $\dB_b({H_+})$ (see Remark~\ref{rem:sgr_funz_ap}). Therefore, a consequence of Proposition~\ref{prop:continuita_misurabilita}-(\ref{prop:contmeas_phi0}) is that, if $\eta \in \dB_b(H_+)$, then the function $\eta^0_P$ appearing in~\eqref{eq:eta0P} belongs to $\dB_0([0,T] \times \overline{H_+}) \cap \dC_b((0,T]\times \overline{H_+})$.
Moreover, if $\eta \in \dC_0(\overline{H_+})$, then $\eta^0_P \in \dC_0([0,T] \times \overline{H_+})$.
\end{remark}

\section{Mild solutions of HJB equations}\label{sec:mildsol}
The purpose of this section is to establish the existence and the uniqueness of the mild solution (see Definition~\ref{definizione:mild_solution} below) to the following HJB equation
\begin{equation}\label{HJB_2}
\left\{
  \begin{aligned}
  &v_t(t,x)=\frac{1}{2}\Tr[Q\dD^2v(t,x)]+ \langle A^* \dD v(t,x), x\rangle\\
	&\qquad \qquad \quad + F(t,x,v(t,x),\dD v(t,x)), & &x\in H_+, \, t\in(0,T],\\
  &v(0,x)=\phi(x), & &x\in H_+, \\
  &v(t,x)=0, & &x\in \partial H_+, \, t\in [0,T],
  \end{aligned}
\right.
\end{equation}
where $T > 0$ is a given time horizon (which will be fixed from now on), $F$ and $\phi$ are given measurable functions. We introduce the following assumption.

\begin{hypothesis}\label{ipotesi_su_F}
The measurable functions $F \colon [0,T]\times H_+\times\R\times H\to \R$ and $\phi \colon H_+\to\R$ verify, for given constants $L, L'>0$, the following:
\begin{enumerate}[(i)]
\item\label{ass:Flip} For any $t\in[0,T]$, $x\in H_+$, $y_1,y_2\in\R$, and $z_1,z_2\in H$, it results
 \[|F(t,x,y_1,z_1)-F(t,x,y_2,z_2)|_\R\leq L(|y_1-y_2|_\R+|z_1-z_2|).\]
\item\label{ass:Flin} For any $t\in[0,T]$, $x\in H_+$, $y\in\R$, and $z\in H$, we have  \[|F(t,x,y,z)|_\R\leq L'(1+|y|_\R+ |z|).\]
\item\label{ass:phi} $\phi\in \dB_b(H_+)$.
\end{enumerate}
\end{hypothesis}

\begin{definition}\label{definizione:mild_solution}
A function $u \colon [0,T]\times \overline{H_+} \to \R$ is a mild solution to the HJB equation~\eqref{HJB_2} if
\begin{enumerate}[(i)]
\item There exists $\eta \in (0,1)$ such that
$u\in \dB^{0,1}_{b,\eta}([0,T]\times \overline{H_+})$;
\item for all $t\in[0,T]$ and all $x\in H_+$, the following equality holds
\begin{equation}\label{mild_form}
u(t,x)=P_t\phi(x) + \int_0^t P_{t-s}[F(s,\cdot,u(s,\cdot),\dD u(s,\cdot)](x) \, \dd s.
\end{equation}
\end{enumerate}
\end{definition}

The following theorem establishes existence and uniqueness of the mild solution to~\eqref{HJB_2}, in the sense of Definition~\ref{definizione:mild_solution}.

\begin{theorem}\label{thm:esis_unic_sol_mild}
Let $\delta \in (0,1)$ be such that Hypotheses~\ref{ipotesi}, \ref{ipotesi_su_Lambda} and~\ref{ipotesi_su_F} are satisfied. Then, Equation~\eqref{HJB_2} has a mild solution
$u\in \dB^{0,1}_{b,\delta}([0,T]\times \overline{H_+})$,
which is unique in this space.
Moreover, $u$ is continuous in $(0,T]\times \overline{H_+}$ and $\dD u$ is continuous in
$(0,T]\times H_+$.
Finally, if $\phi\in \dC_b(H_+)$, $u$ is also continuous in $[0,T]\times H_+$.
\end{theorem}

\begin{proof}
We use the contraction mapping principle on a suitable space to establish the claim.
We consider the space $\sB\coloneqq \dB_0([0,T]\times \overline{H_+})\times \dB_{b,\delta}((0,T]\times H_+;H)$ endowed with the product norm given by the sum of the norms of the factor spaces. To ease notations, we will denote the norm of $\dB_0([0,T]\times \overline{H_+})$ (which is the sup-norm) by $\|\cdot\|_{\dB_0}$.

Let us define the operator $\Upsilon=(\Upsilon_1,\Upsilon_2)$ as
\begin{align*}
\Upsilon_1[u,v](t,x) &\coloneqq P_t\phi (x)+\int_0^t P_{t-s}[F(s,\cdot,u(s,\cdot),v(s,\cdot))](x)\, \dd s, \quad (t,x) \in [0,T] \times \overline{H_+}; \\
\Upsilon_2[u,v](t,x) &\coloneqq \dD P_t\phi (x)+\int_0^t \dD P_{t-s}[F(s,\cdot,u(s,\cdot),v(s,\cdot))](x)\, \dd s, \quad (t,x) \in (0,T] \times H_+.
\end{align*}

To begin with, we need to ensure that the $\Upsilon$ is well defined as a map from $\sB$ into itself.
Let $(u,v)\in \sB$ and consider, first, the function $\Upsilon_1[u,v]$. This is the sum of two functions belonging to $\dB_0([0,T]\times \overline{H_+})$. Indeed, on the one hand, by Proposition~\ref{prop:continuita_misurabilita}-(\ref{prop:contmeas_phi0}) and Remark~\ref{rem:funznullealbordo}, $P_t\phi(x) \in \dB_0([0,T]\times \overline{H_+})$. On the other hand, we readily see that the function
\begin{equation}\label{eq:psi}
\psi(s,x) \coloneqq F(s,x,u(s,x),v(s,x)), \quad (s,x) \in (0,T] \times H_+,
\end{equation}
is Borel measurable on $(0,T] \times H_+$, that $(s,x)\mapsto s^{\delta}\psi(s,x)$ is bounded on $(0,T] \times H_+$, and that, for all $(s,x) \in (0,T] \times H_+$,
\begin{equation}\label{eq:estimateF}
|F(s,x,u(s,x),v(s,x))|_\R\leq L'(1+|u(s,x)|_\R+|v(s,x)|)\leq L'(1+\|u\|_{\dB_b}+s^{-\delta}\|v\|_{\dB_{b,\delta}}).
\end{equation}
This implies that $\psi \in \dB_{b,\delta}((0,T] \times H_+)$, and hence, by Proposition~\ref{prop:continuita_misurabilita}-(\ref{prop:contmeas_psi0}), the map \begin{equation*}
(t,x) \mapsto \int_0^t P_{t-s}[\psi(s, \cdot)](x) \, \dd s = \int_0^t P_{t-s}[F(s,\cdot,u(s,\cdot),v(s,\cdot))](x)\, \dd s
\end{equation*}
is measurable and belongs to $\dB_0([0,T]\times \overline{H_+})$.

Next, considering the function $\Upsilon_2[u,v]$ we see that the first term belongs to $\dB_{b,\delta}((0,T]\times H_+;H)$, thanks to Proposition~\ref{prop:regolarita}, and that the second term is measurable, as a consequence of Proposition~\ref{prop:continuita_misurabilita}-(\ref{prop:contmeas_psi1}) (use the same $\psi$ above to apply this result). We are left to prove that this latter term belongs to $\dB_{b,\delta}((0,T]\times H_+;H)$, a fact that is justified by the following estimate, holding for all $(t,x) \in (0,T] \times H_+$, where we use once more Proposition~\ref{prop:regolarita} and \eqref{eq:estimateF}:
\begin{align*}
&\mathop{\phantom{\leq}} t^\delta\left|\int_0^t \dD P_{t-s}[F(s,\cdot,u(s,\cdot),v(s,\cdot))](x) \, \dd s\right|_\R
\leq L'C t^\delta\int_0^t (t-s)^{-\delta}(1+\|u\|_{\dB_b}+s^{-\delta}\|v\|_{\dB_{b,\delta}}) \, \dd s\\
&\leq L'C (1+\|u\|_{\dB_b}) t^\delta \int_0^t (t-s)^{-\delta} \dd s+ L'C t^\delta \|v\|_{\dB_{b,\delta}} \int_0^t s^{-\delta}(t-s)^{-\delta} \dd s\\
&\leq \frac{L'C}{1-\delta} (1+\|u\|_{\dB_b})T+ L'C\dfrac{\Gamma(1-\delta)^2}{\Gamma(2-2\delta)}\|v\|_{\dB_{b,\delta}}T^{1-\delta} < +\infty,
\end{align*}
where $\Gamma$ is the gamma function.

We proceed, next, to show that $\Upsilon$ is a contraction on $\sB$. To this end, it is convenient to use an equivalent norm on $\sB$, given by the sum of the equivalent norms $\norm{\cdot}_{\beta,\dB_b}$ and $\norm{\cdot}_{\beta,\dB_{b,\delta}}$ on $\dB_0([0,T]\times \overline{H_+})$ and on $\dB_{b,\delta}((0,T]\times H_+;H)$, respectively, defined by:
\begin{align*}
\norm{f}_{\beta,\dB_0} &\coloneqq \sup_{(t,x)\in[0,T]\times \overline{H_+}} \de^{-\beta t}|f(t,x)|_{\R},
&
\norm{f}_{\beta,\dB_{b,\delta}} &\coloneqq \sup_{(t,x)\in (0,T]\times H_+} \de^{-\beta t} t^\delta|f(t,x)|,
\end{align*}
where $\beta > 0$ is a constant to be fixed later in the proof.

We want, now, to find a suitable $\beta>0$ such that the map $\Upsilon=(\Upsilon_1,\Upsilon_2)$ is a contraction on
$(\sB, \norm{\cdot}_{\beta,\dB_0} + \norm{\cdot}_{\beta,\dB_{b,\delta}})$.
We start with an estimate on $\Upsilon_1$. Taking any $(u_1,v_1), \, (u_2,v_2)\in \sB$ and using that $\{P_t\}_{t \geq 0}$ is a semigroup of contractions (cf. Proposition~\ref{prop_P_t}) and \mbox{Hypothesis~\ref{ipotesi_su_F}-(\ref{ass:Flip})}, we have that, for all $(t,x) \in [0,T] \times H_+$,
\begin{align*}
&\mathop{\phantom{\leq}} |\Upsilon_1[u_1,v_1](t,x)-\Upsilon_1[u_2,v_2](t,x)|_\R\\
&= \left|\int_0^t P_{t-s}[F(s,\cdot,u_1(s,\cdot),v_1(s,\cdot))
-F(s,\cdot,u_2(s,\cdot),v_2(s,\cdot))](x) \, \dd s\right|_\R\\
&\leq \int_0^t \|F(s,\cdot,u_1(s,\cdot),v_1(s,\cdot))-F(s,\cdot,u_2(s,\cdot),v_2(s,\cdot))\|_{\dB_b(\overline{H_+})} \, \dd s\\
&\leq L\int_0^t (\|u_1(s,\cdot)-u_2(s,\cdot)\|_{\dB_0}+\|v_1(s,\cdot)-v_2(s,\cdot)\|_{\dB_{b,\delta}}) \, \dd s.
\end{align*}
Now, multiplying and dividing by $\de^{\beta s}$ in the integrals appearing in the last line, we get
\begin{align*}
&\mathop{\phantom{\leq}} |\Upsilon_1[u_1,v_1](t,x)-\Upsilon_1[u_2,v_2](t,x)|_\R\\
&\leq L\int_0^t \left(\de^{\beta s}\|u_1-u_2\|_{\beta,\dB_0}+ s^{-\delta} \de^{\beta s}\|v_1-v_2\|_{\beta,\dB_{b,\delta}} \right) \dd s\\
&\leq L (\|u_1-u_2\|_{\beta,\dB_0}+\|v_1-v_2\|_{\beta,\dB_{b,\delta}})\left[\left(\int_0^t \de^{\beta s} \, \dd s\right) \vee \left(\int_0^t s^{-\delta} \de^{\beta s} \, \dd s\right)\right].
\end{align*}

Clearly $|\Upsilon_1[u_1,v_1](t,x)-\Upsilon_1[u_2,v_2](t,x)|_\R = 0$ on $[0,T] \times \partial H_+$, therefore,
\begin{align*}
&\mathop{\phantom{=}}\|\Upsilon_1[u_1,v_1]-\Upsilon_1[u_2,v_2]\|_{\beta,\dB_0}
=\sup_{t\in[0,T]} \left\{\de^{-\beta t}\|\Upsilon_1[u_1,v_1](t,\cdot)-\Upsilon_1[u_2,v_2](t,\cdot)\|_{\dB_0}\right\}\\
&\leq L C_1(\beta) \left[\|u_1-u_2\|_{\beta,\dB_0}+\|v_1-v_2\|_{\beta,\dB_{b,\delta}}\right],
\end{align*}
where
\begin{align*}
C_1(\beta) &\coloneqq \sup_{t\in[0,T]}\left\{\de^{-\beta t}\left[\left(\int_0^t \de^{\beta s} \, \dd s\right) \vee \left(\int_0^t s^{-\delta} \de^{\beta s} \, \dd s\right)\right]\right\} \\
&=\sup_{t\in[0,T]}\left\{\left(\frac{1-\de^{-\beta t}}{\beta} \right)\vee \left(\int_0^t s^{-\delta} \de^{-\beta (t-s)} \, \dd s\right)\right\}
\leq\frac{1}{\beta} \vee \sup_{t\in[0,T]} \int_0^t s^{-\delta} \de^{-\beta (t-s)} \, \dd s.
\end{align*}
An immediate application of \citep[Proposition~4.21-(iv)]{fabbri:soc} entails that $C_1(\beta)\to 0$, as $\beta\to +\infty$.

We provide, next, an estimate on $\Upsilon_2$, Considering any $(u_1,v_1), \, (u_2,v_2)\in \sB$ and using Proposition~\ref{prop:regolarita} and Hypothesis~\ref{ipotesi_su_F}-(\ref{ass:Flip}), we have that, for all $(t,x) \in [0,T] \times H_+$,
\begin{align*}
&\mathop{\phantom{\leq}} |\Upsilon_2[u_1,v_1](t,x)-\Upsilon_2[u_2,v_2](t,x)|_\R\\
&= \left|\int_0^t \dD P_{t-s}[F(s,\cdot,u_1(s,\cdot),v_1(s,\cdot))-F(s,\cdot,u_2(s,\cdot),v_2(s,\cdot))](x) \, \dd s\right|_\R \\
&\leq 
C \int_0^t (t-s)^{-\delta}\|F(s,\cdot,u_1(s,\cdot),v_1(s,\cdot))-F(s,\cdot,u_2(s,\cdot),v_2(s,\cdot))\|_{\dB_b(\overline{H_+})} \, \dd s\\
&\leq LC \int_0^t (t-s)^{-\delta} (\|u_1(s,\cdot)-u_2(s,\cdot)\|_{\dB_0}+\|v_1(s,\cdot)-v_2(s,\cdot)\|_{\dB_{b,\delta}}) \, \dd s.
\end{align*}
Now, multiplying and dividing by $\de^{\beta s}$ in the integrals appearing in the last line, we get
\begin{align*}
&\mathop{\phantom{\leq}} |\Upsilon_2[u_1,v_1](t,x)-\Upsilon_2[u_2,v_2](t,x)|_\R\\
&\leq LC \int_0^t (t-s)^{-\delta}\left(\de^{\beta s}\|u_1-u_2\|_{\beta,\dB_0}+ s^{-\delta} \de^{\beta s}\|v_1-v_2\|_{\beta,\dB_{b,\delta}} \right) \dd s\\
&\leq LC (\|u_1-u_2\|_{\beta,\dB_0}+\|v_1-v_2\|_{\beta,\dB_{b,\delta}}) \left[\left(\int_0^t (t-s)^{-\delta} \de^{\beta s} \, \dd s\right) \vee \left(\int_0^t (t-s)^{-\delta}s^{-\delta} \de^{\beta s} \, \dd s\right)\right].
\end{align*}
Therefore,
\begin{align*}
&\mathop{\phantom{=}}\|\Upsilon_2[u_1,v_1]-\Upsilon_2[u_2,v_2]\|_{\beta,\dB_{b,\delta}}
=\sup_{t\in[0,T]} \left\{\de^{-\beta t }t^\delta\|\Upsilon_2[u_1,v_1](t,\cdot)-\Upsilon_2[u_2,v_2](t,\cdot)\|_{\dB_b}\right\}\\
&\leq  C_2(\beta) LC \left[\|u_1-u_2\|_{\beta,\dB_0}+\|v_1-v_2\|_{\beta,\dB_{b,\delta}}\right],
\end{align*}
where
\begin{equation*}
C_2(\beta) \coloneqq \sup_{t\in(0,T]}\left\{\left(t^\delta\int_0^t (t-s)^{-\delta} \de^{-\beta (t-s)} \, \dd s\right)\vee \left(t^\delta\int_0^t (t-s)^{-\delta} s^{-\delta} \de^{-\beta (t-s)} \, \dd s\right)\right\}.
\end{equation*}
Applying \citep[Proposition~4.21-(iv) and (v)]{fabbri:soc}, we get that $C_2(\beta)\to 0$, as $\beta\to +\infty$.

By the reasoning above, there exists $\beta_0>0$ such that for $\beta>\beta_0$ we have
\begin{align*}
\|\Upsilon_1[u_1,v_1]-\Upsilon_1[u_2,v_2]\|_{\beta,\dB_0}+\|\Upsilon_2[u_1,v_1]-\Upsilon_2[u_2,v_2]\|_{\beta,\dB_{b,\delta}}\\
\leq \frac{1}{2}\left[\|u_1-u_2\|_{\beta,\dB_0}+\|v_1-v_2\|_{\beta,\dB_{b,\delta}}\right],
\end{align*}
which entails that $\Upsilon$ is a contraction and, therefore, that it has a unique fixed point.

The first component of $\Upsilon$ provides the unique mild solution of \eqref{HJB_2}. Indeed, by Proposition~\ref{prop:regolarita} and Proposition~\ref{prop:continuita_misurabilita}, we get that, for every
$(u,v)\in \sB$, the function $\Upsilon_1[u,v]$ is Fréchet differentiable and $\langle \Upsilon_2[u,v](t,x),h\rangle=\langle \dD \Upsilon_1[u,v](t,x),h\rangle$, for any $(t,x) \in (0,T] \times H_+$ and any $h \in H$.
Therefore, denoting by $(\bar u,\bar v) \in \sB$ the fixed point of $\Upsilon$, which satisfies $\Upsilon[\bar u,\bar v]=(\bar u,\bar v)$, we immediately obtain the following facts: $\bar v(t,x) = \dD\bar u(t,x)$, for all $(t,x) \in (0,T] \times H_+$; $\bar u$ belongs to $\dB_0([0,T]\times \overline{H_+})$; $\bar u$ is Fréchet differentiable, with $\dD\bar u\in \dB_{b,\delta}((0,T]\times H_+;H)$, which implies that $\bar u \in \dB^{0,1}_{b,\delta}([0,T]\times \overline{H_+})$; $\bar u$ verifies \eqref{mild_form}. Hence, by Definition~\ref{definizione:mild_solution}, $\bar u$ is a mild solution to \eqref{HJB_2}.

Furthermore, uniqueness easily follows noticing that any other solution $u^*\in \dB^{0,1}_{b,\delta}([0,T]\times \overline{H_+})$ must be equal to the first component of the fixed point of $\Upsilon$ in $\sB$, i.e., it must hold that $u^*=\bar u$.

The continuity of the solution and of its derivative follows exactly with the same argument of~\cite[Theorem 4.149-(ii)]{fabbri:soc}, exploiting the definition of semigroup $\{P_t\}_{t \geq 0}$ and that semigroup $\{T_t\}_{t \geq 0}$ has a regularizing effect for $t>0$.

To deduce the last assertion we proceed as follows. Let $(t,x), (t_0, x_0) \in [0,T] \times H_+$ and assume, without loss of generality, that $t_0 < t$. Then, recalling~\eqref{eq:psi}, we have that
\begin{multline*}
\abs{u(t,x) - u(t_0,x_0)} \leq \abs{P_t \phi(x) - P_{t_0} \phi(x)} + \abs{P_{t_0} \phi(x) - P_{t_0} \phi(x_0)} + \int_{t_0}^t \abs{P_{t-s}\psi(s,x)} \, \dd s 
\\
+ \int_0^{t_0} \abs{P_{t-s} \psi(s,x) - P_{t_0-s} \psi(s,x)} \, \dd s + \int_0^{t_0} \abs{P_{t_0-s} \psi(s,x) - P_{t_0-s} \psi(s,x_0)} \, \dd s.
\end{multline*}
The result follows from the continuity of $\phi$ and of semigroup $\{P_t\}_{t \geq 0}$ with respect to $t$, from Proposition~\ref{prop_P_t}, and from an application of the dominated convergence theorem.
\end{proof}

\section{Strong solutions of HJB equations}\label{sec:strongsol}
In applications to optimal control it is useful to know that mild solutions to an HJB equation can be approximated by regular solutions, where by regular we mean smooth enough to apply the Itô or the Dynkin formulas. Solutions constructed by this approximating procedure are called \emph{strong solution} (see Definition~\ref{def:strongsol} below for a precise statement).

The idea is to approximate mild solutions to~\eqref{HJB_2} with classical solutions in $\dUC^2_b(H)$. However, it is well-known that $\dUC^2_b(H)$ is not dense in $\dUC_b(H)$, when $\mathrm{dim}(H) = +\infty$ (since unit balls are not compact in this case). As a consequence, we cannot hope for uniform convergence and we need to resort to a different concept of convergence.

We follow the approach of~\citep[Section~4.5]{fabbri:soc} and introduce $\cK$-convergence (cf.~\citep[Definition~B.56 and Definition~4.131]{fabbri:soc}).

\begin{definition}\label{def:Kconv}
Let $\cX$ denote $H$, $\overline{H_+}$, or $H_+$.
A sequence $(f_n)_{n \in \N} \subset \dB_b(\cX)$ is said to be $\cK$-convergent to $f \in \dB_b(\cX)$ and we will write $f = \cK-\lim_{n \to \infty} f_n$, if:
\begin{enumerate}[(i)]
\item $\sup_{n \in \N} \norm{f_n}_{\dB_b(\cX)} < \infty$;
\item $\lim_{n \to \infty} \sup_{x \in K} \abs{f_n(x) - f(x)}=0$, for any compact set $K \subset \cX$.
\end{enumerate}

Similarly, a sequence $(f_n)_{n \in \N} \subset \dB_b([0,T]\times \cX)$ is said to be $\cK$-convergent to $f \in \dB_b([0,T] \times \cX)$ and we will write $f = \cK-\lim_{n \to \infty} f_n$, if:
\begin{enumerate}[(i)]
\item $\sup_{n \in \N} \norm{f_n}_{\dB_b([0,T]\times \cX)} < \infty$;
\item $\lim_{n \to \infty} \sup_{(t,x) \in [0,T] \times K} \abs{f_n(t,x) - f(t,x)}=0$, for any compact set $K \subset \cX$.
\end{enumerate}

Finally, for $\delta \in (0,1)$, we say that a sequence $\{f_n\}_{n\in\N} \subset \dB_{b,\delta}((0,T] \times \cX)$ $\cK$-converges to $f \in \dB_{b,\delta}((0,T] \times \cX)$, and we write $f = \klim_{n \to \infty} f_n$ in $\dB_{b,\delta}((0,T] \times \cX)$, if:
\begin{enumerate}[(i)]
\item $\sup_{n \in \N} \norm{f_n}_{\dB_{b,\delta}((0,T] \times \cX)} < +\infty$;
\item $\lim_{n \to \infty} \sup_{(t,x) \in I_0 \times K} t^\delta \abs{f_n(t,x) - f(t,x)} = 0$, for all compact sets $I_0 \subset (0,T]$ and $K \subset \cX$.
\end{enumerate}
\end{definition}

To give the definition of classical solution, we need to introduce the space
\begin{equation*}
\dUC_b^{2,A}(H_+) \coloneqq \{f \in \dUC_b^2(H_+) \colon
A^*\dD f \in \dUC_b(H_+;H), \, \dD^2 f \in \dUC_b(H_+; \cL_1(H))\}.
\end{equation*}

In this space we introduce the norm
\begin{equation*}
\norm{f}_{\dUC_b^{2,A}(H_+)}
\coloneqq \norm{f}_{\dUC_b(H_+)} + \norm{\dD f}_{\dUC_b(H_+;H)} + \norm{A^*\dD f}_{\dUC_b(H_+;H)}
+ \sup_{x \in H_+}\norm{\dD^2 f(x)}_{\cL_1(H)}.
\end{equation*}

\begin{definition}\label{def:classicsol}
Let $g$ be a given Borel measurable function. A function $u \colon [0,T] \times \overline{H_+} \to \R$ is a classical solution to
\begin{equation}\label{eq:HJBmod}
\left\{
  \begin{aligned}
  &v_t(t,x)=\frac{1}{2}\Tr[Q\dD^2v(t,x)]+\langle x, A^*\dD v(t,x)\rangle \\
	&\qquad \qquad \quad + F(t,x,v(t,x),\dD v(t,x)) + g(t,x), & &x\in H_+, \, t\in(0,T],\\
  &v(0,x)=\phi(x), & &x\in H_+, \\
  &v(t,x)=0, & &x\in \partial H_+, \, t\in[0,T],
  \end{aligned}
\right.
\end{equation}
if:
\begin{enumerate}[(i)]
\item\label{def:classicsol:regol_tempo} $u(\cdot,x) \in \dC^1([0,T])$, for all $x \in H_+$;
\item\label{def:classicsol:regol_spazio} $u(t,\cdot) \in \dUC_b^{2,A}(H_+)$, for any $t \in [0,T]$, and $\sup_{t \in [0,T]} \norm{u(t, \cdot)}_{\dUC_b^{2,A}(H_+)} < +\infty$;
\item\label{def:classicsol:regol_globale} $u \in \dC_0([0,T] \times \overline{H_+})$;
\item\label{def:classicsol:regol_derivate} $\dD u$, $A^*\dD u \in \dC_b([0,T] \times H_+; H)$ and $\dD^2 u \in \dC_b([0,T] \times H_+; \cL_1(H))$;
\item\label{def:classicsol:usoluzione} $u$ satisfies~\eqref{eq:HJBmod} for all $(t,x) \in [0,T] \times \overline{H_+}$.
\end{enumerate}
\end{definition}

We are now ready to introduce the concept of strong solution mentioned at the beginning of this section. Recall that the aim is to show that mild solutions to~\eqref{HJB_2} can be approximated, in the sense of $\cK$-convergence, by classical solutions. Solutions to~\eqref{HJB_2} constructed by this approximating procedure are called $\cK$-strong solutions.
\begin{definition}\label{def:strongsol}
We say that a function $u \colon [0,T] \times \overline{H_+} \to \R$ is a $\cK$-strong solution to~\eqref{HJB_2} if:
\begin{enumerate}[(i)]
\item
There exists $\eta \in (0,1)$ such that $u \in \dB^{0,1}_{b,\eta}([0,T] \times \overline{H_+})$;
\item There exist three sequences $\{u_n\}_{n \in \N} \subset \dC_0([0,T] \times \overline{H_+})$, $\{\phi_n\}_{n \in \N} \subset \dUC_b^{2,A}(H_+) \cap \dC_0(\overline{H_+})$, and $\{g_n\}_{n \in \N} \subset \dB_{b,\delta}((0,T] \times H_+)$ such that:
\begin{enumerate}[(a)]
\item For every $n \in \N$, $u_n$ is a classical solution (in the sense of Definition~\ref{def:classicsol}) to
\begin{equation}\label{eq:HJBmodapprox}
\left\{
  \begin{aligned}
  &v_t(t,x)=\frac{1}{2}\Tr[Q\dD^2v(t,x)]+\langle x, A^*\dD v(t,x)\rangle \\
	&\qquad \qquad \quad + F(t,x,v(t,x),\dD v(t,x)) + g_n(t,x), & &x\in H_+, \, t\in(0,T],\\
  &v(0,x)=\phi_n(x), & &x\in H_+, \\
  &v(t,x)=0, & &x\in \partial H_+, \, t\in[0,T];
  \end{aligned}
\right.
\end{equation}
\item It holds
\begin{equation*}
\begin{dcases}
\klim_{n \to \infty} \phi_n = \phi, &\text{in } \dB_b(H_+), \\
\klim_{n \to \infty} g_n = 0, &\text{in } \dB_{b,\delta}((0,T] \times H_+), \\
\klim_{n \to \infty} u_n = u, &\text{in } \dB_b([0,T] \times H_+), \\
\klim_{n \to \infty} \dD u_n = \dD u, &\text{in } \dB_{b,\delta}((0,T] \times H_+;H). \\
\end{dcases}
\end{equation*}
\end{enumerate}
\end{enumerate}
\end{definition}

\begin{theorem}\label{thm:esist_strong_sol}
Let $\delta\in(0,1)$ be such that Hypotheses~\ref{ipotesi}, \ref{ipotesi_su_Lambda} and \ref{ipotesi_su_F} are satisfied. Let $u$ be the mild solution to~\eqref{HJB_2} and define
\begin{equation}\label{eq:ipF}
f(t,x) \coloneqq F(t,x,u(t,x),\dD u(t,x)), \quad (t,x) \in (0,T]\times H_+.
\end{equation}
Suppose that $\phi \in \dC_b(H_+)$ and that $f$ is continuous on $(0,T] \times H_+$.

Then, the function $u$ is a $\cK$-strong solution to~\eqref{HJB_2}, which is unique among all solutions in $\dB_{b,\delta}^{0,1}([0,T]\times \overline{H_+})$.
\end{theorem}

\begin{proof}
We note, first, that $u$ is continuous both in $[0,T]\times H_+$ and in $(0,T]\times \overline{H}_+$, thanks to Theorem \ref{thm:esis_unic_sol_mild},
and that $f$ is bounded, thanks to
Hypothesis \ref{ipotesi_su_F}-(\ref{ass:Flin}).

To prove the theorem, we need to provide the sequence $\{u_n\}_{n \in \N}$ of classical solutions to~\eqref{eq:HJBmodapprox} that approximate the mild solution $u$ to~\eqref{HJB_2}. To do so, we construct approximating sequences $\{f_n\}_{n \in \N}$ and $\{\phi_n\}_{n \in \N}$ for $f$ and $\phi$, respectively.

Let us fix an orthonormal basis $\cE \subset \cD(A^*)$ of $H$ (this can always be done, as $\cD(A^*)$ is dense). Since $\bar y \in \cD(A^*)$, we can choose $\cE$ so that $\bar y \in \cE$. We denote by $w_j$, $j \geq 2$, the other elements of $\cE$.
Recall that we can identify any element $x \in H$ with the sequence of Fourier coefficients $(x_k)_{k \in \N}$ with respect to the orthonormal basis $\cE$. Let us denote, for each $n \in \N$, the orthogonal projection $P_n$ onto $\mathrm{Span}\{\bar y, w_2, \dots, w_n\}$ (clearly, $P_1$ denotes the projection onto $\mathrm{Span}\{\bar y\}$) and let us define the maps $\Pi_n \colon H \to \R^n$ and $Q_n \colon \R^n \to H$ as
\begin{equation*}
\Pi_n x \coloneqq (x_1, \dots, x_n), \qquad Q_n(x_1, \dots, x_n) \coloneqq x_1 \bar y + \sum_{j=2}^n x_j w_j.
\end{equation*}
Note that $P_n = Q_n \circ \Pi_n$. Let us also consider, for each $n \in \N$, a family of $\dC^\infty$ symmetric mollifiers $\eta_k^n \colon \R^n \to \R$, $k \in \N$, with support in the ball centered at the origin with radius $\frac 1k$.

We approximate $\phi$, first. Since the behaviour of $\phi$ on $\partial H_+$ is not known we do not directly define regularising convolutions for $\phi$, but rather for some approximations of this function. First, we extend $\phi$ to be equal to $0$ for all $x\in \partial H_+$. Then we define, for each $h \in \N$, the functions
\begin{equation*}
\widetilde \phi_h(x) = \widetilde \phi_h(x_1,x') \coloneqq \phi(x_1,x') \chi_{\frac 1h}(x_1), \quad x \in \overline{H_+},
\end{equation*}
where, for any $\epsilon > 0$, $\chi_\epsilon \in \dC^\infty(\R)$ is a function taking values in $[0,1]$, such that $\chi_\epsilon(z) = 0$, if $z \leq \epsilon$, and $\chi_\epsilon(z) = 1$, if $z \geq 2\epsilon$.
We define, for all $x \in H$ and for each $h \in \N$, the following regularizing convolutions of the functions $E \widetilde \phi_h$, where $E$ is the extension operator introduced in~\eqref{def:estensione},
\begin{equation*}
\psi_{k,h}^n(x) \coloneqq \int_{\R^n} E \widetilde \phi_h(Q_n z) \eta^n_k(\Pi_n x - z) \, \dd z = \int_{\R^n} E \widetilde \phi_h(P_n x - Q_n z) \eta^n_k(z) \, \dd z.
\end{equation*}
We immediately obtain that, for all $k, h, n \in \N$,
\begin{equation}\label{eq:psikn_estimate}
\norm{\psi_{k,h}^n}_{\dB_b(H)} \leq \norm{E \widetilde \phi_h}_{\dB_b(H)} \le \norm{\phi}_{\dB_b(H_+)} < +\infty.
\end{equation}
Moreover, we can easily show that $\psi_{k,h}^n$ vanishes on $\partial H_+$. Indeed, for all $x \in \partial H_+$ and recalling that we chose the mollifiers to be symmetric, we have that
\begin{align}
&\mathop{\phantom{=}}\psi^n_{k,h}(x) = \int_{\R^n} E \widetilde \phi_h(P_n x - Q_n z) \eta^n_k(z) \, \dd z \notag
\\
&= \int_{\R^{n-1}} \int_\R E \widetilde \phi_h\biggl(\sum_{j=2}^n (x_j - z_j) w_j - z_1 \bar y\biggr) \eta^n_k(z_1, \dots, z_n) \, \dd z_1 \cdots \dd z_n \notag
\\
&= \int_{\R^{n-1}} \int_{-\infty}^0 \widetilde \phi_h\left(-z_1, x_2 - z_2, \dots, x_n - z_n\right) \eta^n_k(z_1, \dots, z_n) \, \dd z_1 \cdots \dd z_n
\\
&\qquad - \int_{\R^{n-1}} \int_0^{+\infty} \widetilde \phi_h\left(z_1, x_2 - z_2, \dots, x_n - z_n\right) \eta^n_k(z_1, \dots, z_n) \, \dd z_1 \cdots \dd z_n \notag
\\
&= \int_{\R^{n-1}} \int_0^{+\infty} \widetilde \phi_h\left(z_1, x_2 - z_2, \dots, x_n - z_n\right) \eta^n_k(z_1, \dots, z_n) \, \dd z_1 \cdots \dd z_n
\\
&\qquad - \int_{\R^{n-1}} \int_0^{+\infty} \widetilde \phi_h\left(z_1, x_2 - z_2, \dots, x_n - z_n\right) \eta^n_k(z_1, \dots, z_n) \, \dd z_1 \cdots \dd z_n = 0. \label{eq:psikn_boundary}
\end{align}
Let us define, for all $h, n \in \N$, the functions $\xi_{h,n}(x) \coloneqq E \widetilde \phi_h(P_n x)$. Clearly, $\psi_{k,h}^n$ and $\xi_{h,n}$ can be also seen as functions depending on $n$ real variables. Therefore, by standard facts on convolutions (see, e.g.,~\citep{evansgariepy:measth}), we have that, for all $k, h, n \in \N$, $\psi_{k,h}^n \in \dC^\infty(\R^n)$ and that, for any $h,n \in \N$, the sequence $\{\psi_{k,h}^n\}_{k \in \N}$ converges to $\xi_{h,n}$ uniformly on compact subsets of $\R^n$.
Next, for each $h,n \in \N$, we take $k(h,n) \in \N$ such that
\begin{equation*}
\suptwo{x \in H}{\abs{x} \leq n} \abs{\psi^n_{k(h,n),h}(x) - \xi_{h,n}(x)} \leq \dfrac 1n
\end{equation*}
and we set $\psi_{h,n} \coloneqq \psi^n_{k(h,n),h}$. We have that, for any compact set $K \subset H_+$, any $h \in \N$, and any $n \geq \sup_{x \in K} \abs{x}$,
\begin{align*}
\sup_{x \in K} \abs{\psi_{h,n}(x) - E \widetilde \phi_h(x)} 
&\leq \sup_{x \in K} \abs{\psi_{h,n}(x) - \xi_{h,n}(x)} + \sup_{x \in K} \abs{\xi_{h,n}(x) - E \widetilde \phi_h(x)} \\
&\leq \dfrac 1n + \sup_{x \in K} \abs{\xi_{h,n}(x) - \widetilde \phi_h(x)},
\end{align*}
where we used the fact that $E \widetilde \phi_h(x) = \widetilde \phi_h(x)$, for all $x \in H_+$. From this estimate, observing that the set $\{P_n x \colon x \in K, \, n \in \N\} \subset H$ is relatively compact (cf.~\citep[Lemma~B.77]{fabbri:soc}) and using the continuity of $\widetilde \phi_h$ and~\eqref{eq:psikn_estimate}, we get that, for each $h \in \N$,
\begin{equation*}
\klim_{n \to \infty} \psi_{h,n} = \widetilde \phi_h, \qquad \text{in } \dB_b(H_+).
\end{equation*}

Finally, taking the diagonal sequence
\begin{equation*}
\phi_n(x) \coloneqq \psi_{n,n}(x), \qquad x \in \overline{H_+}, \, n \in \N,
\end{equation*}
we easily have that
\begin{equation*}
\klim_{n \to \infty} \phi_n = \phi, \qquad \text{in } \dB_b(H_+),
\end{equation*}
and, applying~\eqref{eq:psikn_boundary}, we get that $\phi_n \in \dC_0(\overline{H_+})$, for all $n \in \N$.

We now turn our attention to the approximation of $f$, which may not belong to $\dC_0((0,T] \times \overline{H_+})$, due to the singularity at $t=0$ and the fact that the behaviour of $f$ on $\partial H_+$ is not known. For this reason, also in this case we define, first, regularising convolutions for some approximations of $f$. First, we extend $f$ to be equal to $0$ for all $(t,x) \in (0,T] \times \partial H_+$. We start approximating $f$ in space, by defining, for each $h \in \N$, the functions
\begin{equation*}
f_h(t,x) = f_h(t,x_1,x') \coloneqq f(t,x_1,x') \chi_{\frac 1h}(x_1), \quad (t,x) \in (0,T] \times \overline{H_+}.
\end{equation*}
Next, we approximate in time by defining, first, the following extensions, for each $h \in \N$ and $x \in \overline{H_+}$,
\begin{equation*}
\overline f_h(t,x) \coloneqq
\begin{dcases}
f_h(t,x), &\text{if } 0 < t \leq T, \\
f_h(T,x), &\text{if } t > T,
\end{dcases}
\end{equation*}
and then by introducing, for each $h \in \N$, the approximations
\begin{equation*}
\widetilde f_h(t,x) \coloneqq \chi_{\frac 1h}(t) \overline f_h(t,x), \quad (t,x) \in \R \times \overline{H_+}.
\end{equation*}
Note that, for each $h \in \N$, $f_h$ is continuous on $(0,T] \times \overline{H_+}$ and vanishes on $\partial H_+$, and hence $\widetilde f_h \in \dC_0(\R \times \overline{H_+})$. This entails that $E\widetilde f_h \in \dC_b(\R \times H)$ (where $E\widetilde f_h$ is defined as in~\eqref{eq:extension_timespace}), which allows us to define, for each $h \in \N$, a sequence $\{f_{h,n}\}_{n \in \N}$ such that
\begin{equation*}
\klim_{n \to \infty} f_{h,n} = E\widetilde f_h, \qquad \text{in } \dB_b(\R \times H).
\end{equation*}
Such a sequence can be constructed proceeding in a similar way as in the first part of this proof, by choosing for each $n$ a suitable family of $\dC^\infty$ mollifiers (see also the proof of~\citep[Lemma~B.78]{fabbri:soc}). In particular, this sequence can be chosen so that
\begin{equation}\label{eq:fhn_estimate}
\suptwo{(t,x) \in \R \times H}{\abs{t} + \abs{x} \leq n} \abs{f_{h,n}(t, x) - E\widetilde f_h(t, P_n x)} \leq \dfrac 1n.
\end{equation}
We consider, next, the diagonal sequence $\{f_{n,n}\}_{n \in \N}$. Using again the fact $E\phi = \phi$ on $H_+$, we have that, for any compact subsets $I_0 \subset (0,T]$, $K \subset H_+$, and any $n \geq \sup_{(t,x) \in I_0 \times K} \{t + \abs{x}\}$,
\begin{align}
&\mathop{\phantom{\leq}} \sup_{(t,x) \in I_0 \times K} t^\delta \abs{f_{n,n}(t,x) - f(t,x)} \notag
\\
&\leq \sup_{(t,x) \in I_0 \times K} t^\delta \abs{f_{n,n}(t,x) - E\widetilde f_n(t, x)} + \sup_{(t,x) \in I_0 \times K} t^\delta \abs{E\widetilde f_n(t, x) - f(t,x)} \notag
\\
&= \sup_{(t,x) \in I_0 \times K} t^\delta \abs{f_{n,n}(t,x) - \widetilde f_n(t, x)} + \sup_{(t,x) \in I_0 \times K} t^\delta \abs{\widetilde f_n(t, x) - f(t,x)}. \label{eq:fn_estimate}
\end{align}
Observe that, by construction, the sequence $\{\widetilde f_{n}\}_{n \in \N}$ $\cK$-converges to $f$ on $\dB_{b,\delta}((0,T] \times H_+)$. Indeed, thanks to the definition of $f_h$, $\widetilde f_h$ converges to $f$ uniformly on compact subsets of $(0,T] \times H_+$, as $h \to +\infty$. Therefore, the last term in the second line of~\eqref{eq:fn_estimate} converges to $0$. To deal with the first term, instead, we note that, if $n$ is large enough, then $\widetilde f_n(t,x) = f(t,x)$, for all $(t,x) \in I_0 \times K$. Therefore, applying also~\eqref{eq:fhn_estimate}, we have that
\begin{align*}
&\mathop{\phantom{\leq}} \sup_{(t,x) \in I_0 \times K} t^\delta \abs{f_{n,n}(t,x) - \widetilde f_n(t, x)}
\\
&\leq \sup_{(t,x) \in I_0 \times K} t^\delta \abs{f_{n,n}(t,x) - \widetilde f_n(t, P_n x)} + \sup_{(t,x) \in I_0 \times K} t^\delta \abs{\widetilde f_n(t, P_n x) - \widetilde f_n(t,x)}
\\
&\leq \dfrac{\sup_{t \in I_0} t^\delta}{n} + \sup_{(t,x) \in I_0 \times K} t^\delta \abs{\widetilde f_n(t, P_n x) - f(t,x)}
\end{align*}
and, thus, noting that the set $\{(t,P_n x) \colon t \in I_0, \, x \in K, \, n \in \N\} \subset \R \times H$ is relatively compact (cf.~\citep[Lemma~B.77]{fabbri:soc}) and using the continuity of $f$, we get that
\begin{equation*}
\klim_{n \to \infty} f_{n,n} = f, \qquad \text{in } \dB_{b,\delta}((0,T] \times H_+).
\end{equation*}
From now on, let us simply denote the diagonal sequence $\{f_{n,n}\}_{n \in \N}$ by $\{f_{n}\}_{n \in \N}$.

Let us define, next, the approximating sequence $\{u_n\}_{n \in \N}$ as follows
\begin{equation*}
u_n(t,x) \coloneqq P_t\phi_n(x) + \int_0^t P_{t-s}[f_n(s,\cdot)](x) \, \dd s, \quad (t,x) \in [0,T] \times \overline{H_+}.
\end{equation*}
We note, first, that $u_n(0,x) = \phi_n(x)$, for all $x \in H_+$, and that, as a consequence of Proposition~\ref{prop:continuita_misurabilita}~\eqref{prop:contmeas_phi0}-\eqref{prop:contmeas_psi0} (see also Remark~\ref{rem:funznullealbordo}), $u_n \in \dC_0([0,T] \times \overline{H_+})$, and hence point~\eqref{def:classicsol:regol_globale} of Definition~\ref{def:classicsol} is verified.
Arguing as in the proof of~\citep[Theorem~4.135]{fabbri:soc}, we deduce that, for each $n \in \N$, $\phi_n \in \dUC_b^{2,A}({H_+})$ and that, for each $t \in (0,T]$, $f_n(t, \cdot) \in \dUC_b^{2,A}({H_+})$. Therefore, by~\citep[Proposition~B.91]{fabbri:soc}, we deduce that points~\eqref{def:classicsol:regol_spazio} and~\eqref{def:classicsol:regol_derivate} of Definition~\ref{def:classicsol} are satisfied.

Next, following the same reasoning of Step~2 of the proof of~\citep[Theorem~4.135]{fabbri:soc}, we get that also points~\eqref{def:classicsol:regol_tempo} and~\eqref{def:classicsol:usoluzione} of Definition~\ref{def:classicsol} are verified, by choosing $g_n(t,x) \coloneqq F(t,x, u_n(t,x), \dD u_n(t,x)) - f_n(t,x)$. Hence, $u_n$ is a classical solution to~\eqref{eq:HJBmodapprox}, for all $n \in \N$.

We are left to check the convergences of the three sequences $\{u_n\}_{n \in \N}$, $\{\phi_n\}_{n \in \N}$, $\{g_n\}_{n \in \N}$ and that $u$ is the unique $\cK$-strong solution to~\eqref{HJB_2}. This can be done exactly in the same way as in Step~3 of the proof of~\citep[Theorem~4.135]{fabbri:soc}.
\end{proof}

\section{Application to a control problem with exit time}\label{sec:optctrl}
In this section we discuss an exit-time optimal control problem. In particular, we study the associated HJB equation and we provide a verification theorem and a result concerning optimal feedback controls.

Fix $T>0$, a real separable Hilbert space $\Xi$, and a complete and separable metric space $U$, representing the \emph{control space}.

We consider a fixed complete filtered probability space $(\Omega, \cF, \bbF \coloneqq (\cF_s)_{s \in [0,T]}, \bbP)$, where filtration $\bbF$ satisfies the usual assumptions, supporting a cylindrical Wiener process $W = (W(s))_{s \in [0,T]}$ on $\Xi$.
We introduce also the following set of \emph{admissible controls}
\begin{equation*}
\cU_{ad} \coloneqq \{u \colon [0,T] \times \Omega \to U \text{ s.t. } \bfu = (u(s))_{s \in [0,T]} \text{ is $\bbF$-progressively measurable}\}.
\end{equation*}

For any fixed $t \in [0,T)$ and any $\bfu \in \cU_{ad}$ we consider the SDE
\begin{equation}\label{eq:controlledSDE}
\begin{dcases}
\dd X(s) = [AX(s) + b(s,X(s),u(s))] \, \dd s + \sqrt Q \, \dd W(s), & s \in [t,T], \\
X(t) = x \in H.
\end{dcases}
\end{equation}

The following assumption will stand from now on.
\begin{hypothesis}\label{assumpt:controlpb}\mbox{}
\begin{enumerate}[(i)]
\item\label{hp:controlpb:AQ} $A$ and $Q$ satisfy Hypotheses~\ref{ipotesi} and \ref{ipotesi_su_Lambda}, for some $\delta \in (0,1)$;
\item\label{hp:controlpb:b} $b\in \dB_b([0,T]\times H\times U;H)$, it is continuous in $(t,x)$ uniformly with respect to $u$, and satisfies
\begin{equation}\label{eq:bLip}
\sup_{s,u \in [0,T] \times U} \abs{b(s,x_1,u) - b(s,x_2,u)} \leq L_b\abs{x_1-x_2}, \quad \forall x_1, \, x_2 \in H.
\end{equation}
\end{enumerate}
\end{hypothesis}

Under these assumptions (see, e.g., \citep[Theorem~1.152]{fabbri:soc}), for any $t \in [0,T)$, $x \in H$, and $\bfu \in \cU_{ad}$, SDE~\eqref{eq:controlledSDE} admits a unique mild solution $X = (X(s;t,x,\bfu))_{s \in [t,T]}$ (in the sense of~\citep[Definition~1.119]{fabbri:soc}) in the class of processes
\begin{equation*}
\cH_2(t,T;H) \coloneqq \{Y \colon [t,T] \times \Omega \to H \text{ progr. meas., s.t. }\sup_{s \in [t,T]} \bbE\abs{Y(s)}^2 < +\infty\}.
\end{equation*}
Moreover, $X$ has continuous trajectories and verifies, for some constant $C > 0$,
\begin{equation*}
\bbE\left[\sup_{s \in [t,T]} \abs{X(s)}^2 \right] \leq C(1+\abs{x}^2).
\end{equation*}

In what follows, we will denote this solution by $(X(s))_{s \in [t,T]}$, when no confusion can arise.

The aim of the optimal control problem is to drive the dynamics of the state process $X$, by choosing a control $\bfu \in \cU_{ad}$ to minimize the \emph{cost functional} $J \colon [0,T] \times \overline{H_+} \times \cU_{ad} \to \R$ defined, for all $t \in [0,T]$, $x \in \overline{H_+}$, $\bfu \in \cU_{ad}$, as
\begin{equation}\label{eq:costfunct}
J(t,x,\bfu) \coloneqq \bbE\left[\int_t^{T \land \tau} \ell(s,X(s;t,x,\bfu),u(s)) \, \dd s + \ind_{T < \tau}\phi(X(T;t,x,\bfu))\right],
\end{equation}
where $\tau \coloneqq \inf\{s \geq t \colon X(s;t,x,\bfu)\in \overline{H_{-}}\}$.

The following assumption will be in force in the remainder of the paper.
\begin{hypothesis}\label{hyp:funzcost}
\mbox{}
\begin{enumerate}[(i)]
\item\label{hyp:l} The running cost function $\ell \colon [0,T] \times \overline{H_+} \times U \to \R$ is measurable and bounded on $[0,T] \times \overline{H_+} \times U$. Moreover, it is continuous in $(t,x) \in [0,T] \times \overline{H_+}$, uniformly with respect to $u \in U$.
\item\label{hyp:phi} $\phi \in \dC_b(H_+)$.
\end{enumerate}
\end{hypothesis}

The value function of the optimal control problem is
\begin{equation}\label{eq:valuefunct}
V(t,x) = \inf_{\bfu \in \cU_{ad}} J(t,x,\bfu), \qquad (t,x) \in [0,T] \times \overline{H_+}.
\end{equation}

\begin{remark}
Clearly, $\tau$ is an $\bbF$-stopping time and depends on $t$, $x$, and the control $\bfu$. However, we will omit this dependence to ease notations.
\end{remark}

We can associate to the exit-time optimal control problem introduced above an HJB equation, that the value function $V$ is expected to satisfy.

To start, let us introduce the \emph{current value Hamiltonian}
\begin{equation}\label{eq:HCV}
F_{CV}(t,x,p,u) \coloneqq \langle p,b(t,x,u)\rangle+\ell(t,x,u), \quad (t,x,p,u) \in [0,T] \times \overline{H_+} \times H \times U,
\end{equation}
and the \emph{Hamiltonian}
\begin{equation}\label{eq:Hamilt}
F(t,x,p) \coloneqq \inf_{u \in U} F_{CV}(t,x,p,u), \quad (t,x,p) \in [0,T] \times \overline{H_+} \times H.
\end{equation}

The HJB equation associated to it is
\begin{equation}\label{prob_per_verification}
\left\{
\begin{aligned}
&v_t(t,x) + \frac{1}{2}\Tr[Q \dD^2 v(t,x)]+\langle A^*\dD v(t,x),x\rangle \\
&\qquad \qquad \quad +F(t,x,\dD v(t,x)) = 0, & &x\in H_+, \, t\in[0,T), \\
&v(T,x)=\phi(x), & &x\in H_+, \\
&v(t,x)=0, & &x\in \partial H_+, \, t\in [0,T].
\end{aligned}
\right.
\end{equation}

\begin{remark}
Note that, while Equation~\eqref{HJB_2} corresponds to an initial value problem, here Equation~\eqref{prob_per_verification} is associated to a terminal value problem, which is a more natural formulation for an optimal control problem. It is clear that all the results in Sections~\ref{sec:mildsol} and \ref{sec:strongsol} can still be applied, simply by reversing time. To be more precise, we give the following definition.
\end{remark}

\begin{definition}\label{def:mildstrongoptctrl}
A function $v \colon [0,T]\times \overline{H_+} \to \R$ is a mild solution (resp. $\cK$-strong solution) to the HJB equation~\eqref{prob_per_verification} if the function $z(t,x) \coloneqq v(T-t,x)$, $(t,x) \in [0,T]\times \overline{H_+}$, is a mild solution (resp. $\cK$-strong solution) to~\eqref{HJB_2} in the sense of Definition~\ref{definizione:mild_solution} (resp. Definition~\ref{def:strongsol}).
\end{definition}

We want now to verify that the HJB equation~\eqref{prob_per_verification} has a unique mild and $\cK$-strong solution. To do so, we need to check that the Hamiltonian $F$ verifies appropriate assumptions.
\begin{proposition}\label{prop:F}
Under Hypothesis~\ref{hyp:funzcost}-(\ref{hyp:l}), the Hamiltonian $F$ is a measurable function and satisfies Hypothesis~\ref{ipotesi_su_F}-(\ref{ass:Flip})-(\ref{ass:Flin}). Moreover, $F$ is continuous on $[0,T] \times \overline{H_+} \times H$.
\end{proposition}

\begin{proof}
Since $F_{CV}$ is clearly a measurable function and $U$ is separable, it follows that $F$ is a measurable function, too.

To check Hypothesis~\ref{ipotesi_su_F}-(\ref{ass:Flip}), let $(t,x)\in [0,T]\times \overline{H_+}$ and $p_1$, $p_2\in H$. Then,
\begin{equation*}
\abs{F(t,x,p_1)-F(t,x,p_2)} \leq \sup_{u\in U}\{\langle b(t,x,u),p_1-p_2\rangle\}\leq \abs{b}_\infty \abs{p_1-p_2}.
\end{equation*}

Finally, to check Hypothesis~\ref{ipotesi_su_F}-(\ref{ass:Flin}), let $(t,x)\in [0,T]\times \overline{H_+}$ and $p \in H$. Then,
\begin{equation*}
\abs{F(t,x,p)} \leq \abs{b}_\infty \abs{p} + \abs{\ell}_\infty.
\end{equation*}

To establish the continuity of $F$ consider $(t,x,p) \in [0,T] \times \overline{H_+} \times H$ and a sequence $(t_n,x_n,p_n)\to(t,x,p)$. Then,
\begin{align*}
\abs{F(t_n,x_n,p_n)&-F(t,x,p)}\leq \|b\|_\infty|p_n-p|+\\
&+\sup_{u\in U}\{|\langle b(t_n,x_n,u),p\rangle-\langle b(t,x,u),p\rangle+\ell(t_n,x_n,u)-\ell(t,x,u)|\},
\end{align*}
and hence the left hand side tends to $0$, since $b$ and $\ell$ are continuous in $(t,x)$, uniformly with respect to $u \in U$.
\end{proof}

\begin{theorem}\label{th:mildstronguniq}
Equation~\eqref{prob_per_verification} has a unique mild and $\cK$-strong solution in the sense of Definition~\ref{def:mildstrongoptctrl}.
\end{theorem}

\begin{proof}
First of all, we observe that, thanks to Assumption~\ref{hyp:funzcost}-(\ref{hyp:phi}), the terminal cost $\phi$ satisfies~\ref{ipotesi_su_F}-(\ref{ass:phi}). Moreover, by Proposition~\ref{prop:F}, the Hamiltonian $F$ satisfies \ref{ipotesi_su_F}-(\ref{ass:Flip})-(\ref{ass:Flin}). Finally, $Q$ and $A$, verify Hypotheses~\ref{ipotesi} and \ref{ipotesi_su_Lambda}, thanks to Hypothesis~\ref{assumpt:controlpb}.

Therefore, we can apply Theorem~\ref{thm:esis_unic_sol_mild} and conclude that Equation~\eqref{prob_per_verification} has a unique mild solution in the sense of Definition~\ref{def:mildstrongoptctrl}.

Finally, since $\phi \in \dC_b(H_+)$ by Hypothesis~\ref{hyp:funzcost}-\eqref{hyp:phi} and $F$ is continuous on $[0,T]\times \overline{H}_+ \times H$ thanks to Proposition~\ref{prop:F}, we can apply Theorem~\ref{thm:esist_strong_sol}
to deduce that Equation~\eqref{prob_per_verification} has a unique $\cK$-strong solution in the sense of Definition~\ref{def:mildstrongoptctrl}.
\end{proof}

Our next aim is to establish a verification theorem. To do so, we need to show that the following fundamental identity holds.

\begin{lemma}\label{lem:fundid}
\label{lm4:fundidcontrol}
Suppose that Hypotheses~\ref{assumpt:controlpb} and \ref{hyp:funzcost} hold.
Let $v$ be the unique mild and \mbox{$\cK$-strong} solution to~\eqref{prob_per_verification}. Then, for every $t \in [0,T]$, $x \in \overline{H_+}$, and $\bfu \in \cU_{ad}$, it holds that
\begin{multline}\label{eq:fundid}
v(t,x) = J(t,x,\bfu) \\
- \bbE \int_t^{T \land \tau} \left[ F_{CV}\left(s,X(s),\dD v(s,X(s),u(s))\right) -
F\left(s,X(s),\dD v(s,X(s))\right)\right]ds.
\end{multline}
where $X = (X(s;t,x,\bfu))_{s \in [t,T]}$ is the mild solution to~\eqref{eq:controlledSDE}.
\end{lemma}

\begin{proof}
Let $z(t,x) \coloneqq v(T-t,x)$, $(t,x) \in [0,T]\times \overline{H_+}$, so that, according to Definition~\ref{def:mildstrongoptctrl}, $z \in \dB^{0,1}_{b,\delta}([0,T]\times \overline{H_+})$ is the unique $\cK$-strong solution to~\eqref{HJB_2}, in the sense of Definition~\ref{def:strongsol}.

Therefore, there exist three sequences $\{z_n\}_{n \in \N} \subset \dC_b([0,T] \times \overline{H_+})$, $\{\phi_n\}_{n \in \N} \subset \dUC_b^{2,A}(\overline{H_+})$, and $\{g_n\}_{n \in \N} \subset \dB_{b,\delta}((0,T] \times H_+)$, such that, for every $n \in \N$, $z_n$ is a classical solution to~\eqref{eq:HJBmodapprox}, in the sense of Definition~\ref{def:classicsol}. Moreover, we have that
\begin{equation*}
\begin{dcases}
\klim_{n \to \infty} \phi_n = \phi, &\text{in } \dB_b(H_+), \\
\klim_{n \to \infty} g_n = 0, &\text{in } \dB_{b,\delta}((0,T] \times H_+), \\
\klim_{n \to \infty} z_n = z, &\text{in } \dB_b([0,T] \times {H_+}), \\
\klim_{n \to \infty} \dD z_n = \dD z, &\text{in } \dB_{b,\delta}((0,T] \times {H_+};H). \\
\end{dcases}
\end{equation*}

It is clear that~\eqref{eq:fundid} holds for any $t \in [0,T]$, whenever $x \in \partial H_+$. Indeed, in this case $\tau = t$, $\bbP$-a.s., and both the left and the right hand sides of~\eqref{eq:fundid} are equal to $0$. Therefore, we can restrict our attention to the case where $(t,x) \in [0,T] \times H_+$.

We show, first, that~\eqref{eq:fundid} holds for the functions $v_n(t,x) \coloneqq z_n(T-t,x)$, $(t,x) \in [0,T]\times \overline{H_+}$. It is clear that, for all $n \in \N$, $v_n$ and $z_n$ enjoy the same regularity properties, as listed in Definition~\ref{def:classicsol}, and that $v_n$ satisfies
\begin{equation*}
\left\{
  \begin{aligned}
  &\partial_t v_n(t,x)+\frac{1}{2}\Tr[Q\dD^2v_n(t,x)]+\langle x, A^*\dD v_n(t,x)\rangle \\
	&\qquad \qquad \quad + F(t,x,v_n(t,x),\dD v_n(t,x)) + g_n(t,x)=0, & &x\in H_+, \, t\in [0,T),\\
  &v_n(T,x)=\phi_n(x), & &x\in H_+, \\
  &v_n(t,x)=0, & &x\in \partial H_+, \, t\in[0,T],
  \end{aligned}
\right.
\end{equation*}
where $\partial_t$ denotes the time-derivative. It is also clear that
\begin{equation*}
\begin{dcases}
\klim_{n \to \infty} v_n = v, &\text{in } \dB_b([0,T] \times H_+), \\
\klim_{n \to \infty} \dD v_n = \dD v, &\text{in } \overline{\dB}_{b,\delta}([0,T) \times H_+;H), \\
\end{dcases}
\end{equation*}
where\footnote{For more details on this function space, see~\citep[Definition~4.24]{fabbri:soc}.}
\begin{equation*}
\overline{\dB}_{b,\delta}([0,T) \times H_+;H) \coloneqq \{f \colon [0,T) \times H_+ \to H, \text{ s.t. } f(T-\cdot, \cdot) \in \dB_{b,\delta}((0,T] \times H_+;H)\}.
\end{equation*}

Thanks to these facts and to Hypothesis~\ref{assumpt:controlpb}, we can apply the Dynkin formula of
\citep[Equation~(1.109)]{fabbri:soc} to the process $\left(v_n(s,X(s))\right)_{s\in[t,T]}$. For $\epsilon > 0$, let us define the stopping time $\tau_\epsilon \coloneqq (\tau-\epsilon) \vee t$. Then,
\begin{align}
&\mathop{\phantom{=}}\bbE[v_n(T \land \tau_\epsilon, X(T \land \tau_\epsilon))] = v_n(t,x)
\notag \\
&\quad +\bbE \int_t^{T \land \tau_\epsilon} \bigl[\partial_t v_n(s, X(s)) + \prodscal{X(s), A^* \dD v_n(s,X(s))} + \prodscal{\dD v_n(s,X(s)), b(s,X(s),u(s))}\bigr] \, \dd s
\notag \\
&\quad + \bbE \int_t^{T \land \tau_\epsilon} \dfrac 12 \Tr[Q \dD^2 v_n(s,X(s))] \, \dd s
\notag \\
&=\bbE \int_t^{T \land \tau_\epsilon} \prodscal{\dD v_n(s,X(s)), b(s,X(s),u(s))} \, \dd s
\notag \\
&\quad -\bbE \int_t^{T \land \tau_\epsilon} \bigl[F(s,X(s),\dD v_n(s,X(s))) + g_n(s,X(s))\bigr] \, \dd s. \label{eq:Dyn1}
\end{align}
Clearly, on $\{T < \tau_\epsilon\}$ we have that $T \land \tau_\epsilon = T$, whence $X(T \land \tau_\epsilon) = X(T) \in H_+$, and $v_n(T \land \tau_\epsilon, X(T \land \tau_\epsilon)) = v_n(T,X(T)) = \phi(X(T))$. 

Therefore we have, $\bbP$-a.s., 
$$\bbE[v_n(T \land \tau_\epsilon, X(T \land \tau_\epsilon))] = \bbE[\ind_{T < \tau_\epsilon} v_n(T \land \tau_\epsilon, X(T \land \tau_\epsilon))] 
+ \bbE[\ind_{T \ge  \tau_\epsilon} v_n(\tau_\epsilon, X(\tau_\epsilon))] 
$$ $$
= \bbE[\ind_{T < \tau_\epsilon} \phi_n(X(T))] + \bbE[\ind_{T \ge  \tau_\epsilon} v_n(\tau_\epsilon, X(\tau_\epsilon))].
$$
Taking into account this equality, adding $J(t,x,\bfu)$ on both sides of~\eqref{eq:Dyn1}, and rearranging the terms, we get
\begin{align*}
&\mathop{\phantom{+}} v_n(t,x) = J(t,x,\bfu) + \bbE[\ind_{T < \tau_\epsilon} \{\phi_n(X(T)) - \phi(X(T))\}] + \bbE \int_t^{T \land \tau_\epsilon} g_n(s,X(s)) \, \dd s
\\
&- \bbE \int_t^{T \land \tau_\epsilon} \bigl[F_{CV}(s,X(s),\dD v_n(s,X(s)), u(s)) - F(s,X(s), \dD v_n(s,X(s)))\bigr] \, \dd s 
\\ & + \bbE[\ind_{T \ge  \tau_\epsilon} v_n(\tau_\epsilon, X(\tau_\epsilon))].
\end{align*}
Using the $\cK$-convergences previously recalled and the dominated convergence theorem, we can take the limit as $n \to +\infty$ in the last equality to get
\begin{align*}
&\mathop{\phantom{+}} v(t,x) = J(t,x,\bfu) + \bbE[\ind_{T \ge  \tau_\epsilon} v(\tau_\epsilon, X(\tau_\epsilon))].
\\
&- \bbE \int_t^{T \land \tau_\epsilon} \bigl[F_{CV}(s,X(s),\dD v(s,X(s)), u(s)) - F(s,X(s), \dD v(s,X(s)))\bigr] \, \dd s.
\end{align*}
To conclude, it suffices to observe that $\tau_\epsilon \to \tau$, $\bbP$-a.s., as $\epsilon \to 0$, and that $X(\tau) \in \partial H_+$, and hence $\bbE [\ind_{T \ge  \tau} v(\tau, X(\tau) )] =0$. Therefore, using the monotone convergence theorem and the dominated convergence theorem, we get the claim taking the limit as $\epsilon \to 0$ in the last equality.
\end{proof}

We can now state a verification theorem. Its proof is standard and a straightforward consequence of Lemma~\ref{lem:fundid}; therefore, we omit it (for more details see, e.g., \citep[Theorem~2.36]{fabbri:soc}).
\begin{theorem}\label{th:ver}
Suppose that Hypotheses~\ref{assumpt:controlpb} and \ref{hyp:funzcost} hold.
Let $v$ be the mild and $\cK$-strong solution of~\eqref{prob_per_verification}. Then,
\begin{equation*}
v(t,x) \leq V(t,x), \quad \forall \, (t,x) \in [0,T] \times \overline{H_+}.
\end{equation*}

Suppose, moreover, that $\bfu^* \in \cU_{ad}$ is such that, for fixed $(t,x) \in [0,T] \times \overline{H_+}$, the pair $(X^*, \bfu^*)$, where $X^* \coloneqq (X(s; t,x,\bfu^*))_{s \in [t,T]}$, verifies
\begin{equation}\label{eq:optfb}
u^*(s) \in \arg\min_{u \in U}
F_{CV}(s, X^*(s), \dD v(s, X^*(s)),u),
\end{equation}
for almost every $s\in [t,T]$ and $\mathbb{P}$-a.s.
Then, the pair $(X^*, \bfu^*)$ is optimal at $(t,x)$ and $v(t,x) = V(t,x)$.
\end{theorem}

We conclude our paper with a final result concerning optimal feedback controls, which is a corollary of Theorem~\ref{th:ver}.

Let $v$ be the unique mild and $\cK$-strong solution to~\eqref{prob_per_verification} and let us define the set-valued function
\begin{equation}\label{eq:fbmap}
\Phi(t,x) \coloneqq \arg\min_{u \in U}
F_{CV}(t, x, \dD v(t, x),u), \quad (t,x) \in (0,T) \times \overline{H_+}.
\end{equation}
Function $\Phi$ is the feedback map for our optimal control problem and the associated \emph{closed loop equation} is the stochastic differential inclusion
\begin{equation*}
\begin{dcases}
\dd X(s) \in [AX(s) + b(s,X(s),\Phi(s,X(s)))] \, \dd s + \sqrt Q \, \dd W(s), & s \in [t,T], \\
X(t) = x \in H.
\end{dcases}
\end{equation*}

Here is our final result. Its proof is omitted, since it is a direct consequence of Theorem~\ref{th:ver} (for more details see, e.g., \citep[Corollary~2.38]{fabbri:soc}).
\begin{corollary}
Fix $(t,x) \in [0,T) \times \overline{H_+}$ and let $v$ be the unique mild and $\cK$-strong solution to~\eqref{prob_per_verification}. Suppose that Hypotheses~\ref{assumpt:controlpb} and \ref{hyp:funzcost} hold and assume that, on $[t, T) \times \overline{H_+}$, the feedback map $\Phi$, defined in~\eqref{eq:fbmap}, admits a measurable selection $\phi_t \colon [t,T) \times \overline{H_+} \to U$ such that the closed loop equation
\begin{equation}\label{eq:cle}
\begin{dcases}
\dd X(s) = [AX(s) + b(s,X(s),\phi_t(s,X(s)))] \, \dd s + \sqrt Q \, \dd W(s), & s \in [t,T], \\
X(t) = x,
\end{dcases}
\end{equation}
has a mild solution (in the sense of~\citep[Definition~1.119]{fabbri:soc}), denoted by $X_{\phi_t} \coloneqq (X_{\phi_t}(s;t,x))_{s \in [t,T]}$.
Define also $\bfu_{\phi_t} \coloneqq (u_{\phi_t}(s))_{s \in [t,T]}$, where $u_{\phi_t}(s) = \phi_t(s, X_{\phi_t}(s))$.

If $\bfu_{\phi_t} \in \cU_{ad}$, then the pair $(\bfu_{\phi_t}, X_{\phi_t})$ is optimal at $(t, x)$ and $v(t, x) = V(t, x)$. 	
Moreover, if $\Phi(t, x)$ is always a singleton\footnote{This happens, e.g., when $b$ is linear and $\ell$ is quadratic and strictly convex.} and the mild solution of~\eqref{eq:cle} is unique, then $\bfu_{\phi_t}$ is the unique optimal control.
\end{corollary}

\begin{example}
We consider here a simplified version of the spatial economic growth problem of~\citep{BFFGjoeg} in a stochastic and finite time horizon setting (see also~\citep{boucekkine:spatialAK}, \citep{CalviaFedericoGozzi21}, and \citep{GozziLeocata22}).
We provide this example as a motivation for our paper and for future research in exit time optimal control problems and optimal control problems with state space constraints, as explained in the Introduction. Indeed, as we will see in a moment, our example is related to both kind of dynamic optimization problems and it is formulated in a slightly more general setting than the one of Section~\ref{sec:optctrl}. This will give us the chance to point out possible future developments of our research. In the following, we denote by $\bfS^1$ the unit circle and we set $H \coloneqq \dL^2(\bfS^1)$, that is, we consider $H$ to be the set of square-integrable functions on $\bfS^1$.

We present, first, a finite time horizon version of the model studied in~\citep{BFFGjoeg}, which is formulated in a deterministic setting. In this paper the state variable is the capital stock $k(t,\xi)$ and the control variable, which is assumed to be non-negative, is the consumption flow $c(t,\xi)$; both variables depend on time $t \in [0,T]$, where $T > 0$ is a fixed time horizon, and on the spatial position $\xi\in \bfS^1$. Given suitable bounded measurable data $A \colon \bfS^1\to \R_+$, $k_0 \colon \bfS^1\to \R_+$, and $t \in [0,T]$, we consider the state equation
\begin{equation}\label{eq:K}
\begin{cases}
\dd k(s,\xi) = \left[\dfrac{\partial^2}{\partial \xi^2} k(s,\xi)
+ A(\xi) k(s,\xi) - c(s,\xi) \right] \dd s, &
(s,\xi) \in (t,T] \times \bfS^1, \\
k(t,\xi) = k_0(\xi), & \xi \in \bfS^1.
\end{cases}
\end{equation}
We call $k^{t,k_0,c}$ the unique solution to such state equation.
We require the state to satisfy the constraint $k(s,\cdot) \in H_+$, for all $s \in [t,T]$, where
\begin{equation}\label{eq:vincoloBFFG}
H_+ \coloneqq \left\{k \in \dL^2(\bfS^1) \colon \int_0^{2\pi} k(\xi) \bar y(\xi) \, \dd \xi > 0\right\},
\end{equation}
and $\bar y$ is a strictly positive eigenvector (unique up to a multiplication by a scalar, as explained in~\cite[Section 4]{CalviaFedericoGozzi21}) associated to the operator $L \colon D(L) \subseteq \dL^2(\bfS^1)\to \dL^2(\bfS^1)$, defined as
\begin{equation*}
[Lk](\xi)=\dfrac{\partial^2}{\partial \xi^2} k(\xi)
 + A(\xi) k(\xi), \quad x \in \bfS^1.
\end{equation*}
The objective functional to be maximized in this finite time horizon setting is
\begin{equation*}
J_0(t,k_0;c) \coloneqq
\int_{t}^T \int_0^{2\pi}
U_0(s,k^{t,k_0,c}(s,\xi),c(s,\xi)) \, \dd s, \quad t \in [0,T],
\end{equation*}
where $U_0 \colon [0,T] \times \R \times \R_+ \to \R$ is a suitable running gain function and controls $c$ are chosen in the set
\begin{equation*}
\cA_{ad}^D(t, k_0) \coloneqq \{c \in \dL^1((t,T]; \dL^2({\bf S^1}; [0,+\infty))) \colon k^{t,k_0,c}(s, \cdot) \in H_+, \, \text{for a.a. } s \in (t,T]\},
\end{equation*}
where the dependence of this set on $(t,k_0) \in [0,T] \times H_+$, is due to the fact that the solution to~\eqref{eq:K} depends on this pair.

If we assume that $U_0(s,k,c)\geq 0$ everywhere and that $U_0(s,k,0) = 0$ for all $s\in [0,T]$, $k\in \R$, then the state constrained problem above is equivalent to an exit-time problem from the half-plane $H_+$ defined in~\eqref{eq:vincoloBFFG}.
To see this equivalence, let us define the objective functional associated to the exit-time problem
\begin{equation*}
J(t,k_0;c) \coloneqq \int_t^{T \land \tau} U_0(s,k(s,\xi),c(s,\xi)) \, \dd s, \quad t \in [0,T],
\end{equation*}
where controls $c$ are chosen in the set $\dL^1_{loc}\left((t,T]; \dL^2(\bfS^1; [0,+\infty))\right)$ and
\begin{equation*}
\tau \coloneqq \inf\{s \in (t,T] \text{ s.t. } k(s, \cdot) \in \partial H_+\}.
\end{equation*}
Let us fix $t \in [0,T]$ and $k_0 \in H_+$. On the one hand, if $c \in \cA_{ad}(t, k_0)$, then \emph{a fortiori} $c$ is admissible for the exit-time problem and $\tau > T$. Therefore, the functionals $J_0$ and $J$ coincide for all $c \in \cA_{ad}(t, k_0)$, and hence the value of the state constrained optimal control problem is not greater than the value of the exit-time problem. On the other hand, let us consider an arbitrary control $c \in \dL^1_{loc}\left((t,T]; \dL^2(\bfS^1; [0,+\infty))\right)$ and the constant control equal to zero, denoted by $\bfzero$, which verifies $\bfzero \in \cA_{ad}(t, k_0)$. Let $\tau$ be the first time at which the solution to~\eqref{eq:K} with control $c$ reaches the boundary of $H_+$ and let $\bar c$ be the following control, for $s \in (t,T]$,
\begin{equation*}
\bar c(s) \coloneqq
\begin{dcases}
c(s), &\text{if } s < \tau \land T, \\
\bfzero, &\text{if } s \geq \tau \land T.
\end{dcases}
\end{equation*}
Then, $\bar c \in \cA_{ad}(t, k_0)$ and it is also admissibile for the exit-time problem. Using the assumptions on $U_0$, we can immediately show that the functionals $J_0$ and $J$ coincide with this control, and hence that the value of the exit-time problem is not greater than the value of the constrained optimal control problem. Thus, the two optimization problems are equivalent.

The exit-time optimal control problem can be formulated in an infinite-dimensional stochastic setting as follows.
We call the state variable $X(s) \coloneqq k(s,\cdot)$ and the control variable $u(s) \coloneqq c(s,\cdot)$. We consider a cylindrical Wiener process $W$, as in the beginning of this section, $Q \in \cL^+_1(H)$, and we suppose that $X$ satisfies
\begin{equation}\label{eq:K_abstract}
\begin{dcases}
\dd X(s)=[LX(s)-u(s)] \dd s + \sqrt{Q} \, \dd W(s), & s\in [t,T],\\
X(t)=x\in \dL^2(\bfS^1).
\end{dcases}
\end{equation}
Equation~\eqref{eq:K_abstract} is precisely of the form given in~\eqref{eq:controlledSDE}. We preferred to denote by $L$ (instead of $A$) the operator associated to the linear part of this SDE to maintain the standard notation of datum $A$ in~\eqref{eq:K}, which is related to the economic interpretation of this model (see, e.g., \citep{BFFGjoeg} for more details). It is possible to verify that $L$ generates a $C_0$-semigroup and that it is self-adjoint, so that Hypoteses~\ref{ipotesi}-\eqref{ass:Agenerator}-\eqref{ass:alpha} are verified. We also suppose that $Q$ verifies all other points of Hypothesis~\ref{ipotesi}.
The possibily nonlinear term $b$ appearing in the drift of SDE~\eqref{eq:K_abstract} is simply given by $b(s,x,u) = -u$.
The set of admissible controls is, for given $M>0$,
\begin{align*}
\cA_{ad}^S(t,x) \coloneqq
\Bigl\{ &u \in
\dL^1((t,T]\times \Omega; \dL^2({\bf S^1}; [0,+\infty))) \colon
u \text{ is progressively measurable},
\\
&u(s,\omega)(\xi) \leq M \text{ for all } (s,\xi,\omega)\in [t,T]\times \bfS^1 \times \Omega,
\\
&X^{t,x,u}(s) \in H_+, \, \text{for a.a. } s \in (t,T], \; \bbP\text{-a.s.}\Bigr\},
\end{align*}
Note that we choose bounded controls so that Hypothesis~\ref{assumpt:controlpb}-\eqref{hp:controlpb:b} is verified. 

Next, we consider the functional
\begin{equation}\label{eq:J1}
J_1(t,x;u) \coloneqq
\E\left[
\int_{t}^{T \land \tau}
\ell(s,X(s),u(s)) \, \dd s
\right] ,
\end{equation}
where $\tau$ is the first exit time from the half plane $H_+$ defined in~\eqref{eq:vincoloBFFG} and, for any $s \in [0,T]$, $k \colon \bfS^1 \to \R$, $c \colon \bfS^1 \to \R_+$ (provided that the integral below is well-defined),
\begin{equation*}
\ell(s,k,c) \coloneqq
\int_0^{2\pi} U_0(s,k(\xi),c(\xi)) \, \dd \xi.
\end{equation*}
We assume that $U_0$ is such that Hypothesis~\ref{hyp:funzcost}-\eqref{hyp:l} is verified. Hence, under the above setting our main results
(in particular Theorems \ref{thm:esis_unic_sol_mild}, \ref{thm:esist_strong_sol}, Theorem \ref{th:ver})
apply.

We conclude our example with an important note. We must be clear on the fact that the original problem studied in the economic literature (see, e.g., \citep{BFFGjoeg}
and \cite{GozziLeocata22}) considers an infinite time horizon setting, unbounded controls and functions $U_0$ representing utility functions of the \emph{Constant Elasticity of Substitution}-type, i.e.,
\begin{equation*}
U_0(c) = \dfrac{c^{1-\sigma}}{1-\sigma}, \quad \text{with } \sigma>1.
\end{equation*}
Although it seems possible to extend our result to an infinite time horizon setting and to the case of unbounded controls, dealing with utility functions $U_0$ of such type seems quite difficult at the moment.
All such extensions will be the subject of future research.
\end{example}

\bibliographystyle{plainnat}
\bibliography{Bibliography}

\end{document}